\documentclass{amsart}

\usepackage{amssymb}
\usepackage{graphicx}
\usepackage{verbatim}
\usepackage{stmaryrd}
\usepackage{enumitem}

\newtheorem{remark}{Remark}
\newtheorem{lemma}{Lemma}
\newtheorem{theorem}{Theorem}
\newtheorem{corollary}{Corollary}
\newtheorem{definition}{Definition}

\newcommand{\PERMS}[2][]{\mathbb{P}^{#1}_{#2}}
\newcommand{\MISSING}[2][]{\mathbb{M}^{#1}_{#2}}

\newcommand{\SigmaCycles}[1]{S(#1)}
\newcommand{\AltCycles}[1]{A(#1)}

\newcommand{\CW}{\includegraphics[height=0.5em]{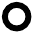}}
\newcommand{\CCW}{\includegraphics[height=0.5em]{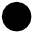}}
\newcommand{\dart}[2]{#1 \rightarrow #2}

\newcommand{\xor}{\oplus}
\newcommand{\setdiff}{\backslash}
\newcommand{\nset}[1]{\llbracket #1 \rrbracket}
\newcommand{\quotient}[1]{#1\slash{\sim}}

\newcommand{\walk}[1]{\mathsf{walk}(#1)}
\newcommand{\edges}[1]{\mathsf{edges}(#1)}
\newcommand{\walkName}{\mathsf{walk}}
\newcommand{\edgesName}{\mathsf{edges}}
\newcommand{\parent}[1]{\mathsf{parent}(#1)}
\newcommand{\parentWalk}{\mathsf{parent}}
\newcommand{\sink}[2]{\mathsf{sink}(#1,#2)}

\newcommand{\FixSubset}[3][]{\mathsf{F}_{#1}(#2,#3)}
\newcommand{\YCover}[1]{\mathbf{Y}\!_2(#1)}
\newcommand{\YCycle}[1]{\mathbf{Y}\!_1(#1)}
\newcommand{\ACover}[1]{A_2(#1)}
\newcommand{\ACycle}[1]{A_1(#1)}
\newcommand{\WilfCycle}[1]{\mathsf{Wilf}_1(#1)}
\newcommand{\WilfCover}[1]{\mathsf{Wilf}_2(#1)}
\newcommand{\WilfC}[1]{\mathsf{W}_1(#1)}
\newcommand{\WilfP}[1]{\mathsf{W}_2(#1)}

\newcommand{\TheDigraph}[1]{\mathcal{G}(#1)}
\newcommand{\TheCayleyDigraph}[1]{\protect \overrightarrow{\textsc{cayley}}(\SymGroup{#1}, \{\sigma, \tau\})}
\newcommand{\TheCayleyGraph}[1]{\textsc{cayley}(\SymGroup{#1}, \{\sigma, \tau\})}
\newcommand{\CayleyG}[1]{\protect \overrightarrow{\textsc{cayley}}(\SymGroup{#1}, G)}
\newcommand{\SymGroup}[1]{\mathbb{S}_{#1}}
\newcommand{\SigmaEdges}{\mathcal{E}_{\sigma}}
\newcommand{\TauEdges}{\mathcal{E}_{\tau}}

\newcommand{\Wilf}[1]{\mathsf{Wilf}(#1)}
\newcommand{\Wheeel}[1]{\mathsf{Wheeel}(#1)}

\newcommand{\TheCycle}[1]{\mathcal{C}_{1}(#1)}
\newcommand{\TheCover}[1]{\mathcal{C}_{2}(#1)}
\newcommand{\ThePath}[1]{\mathcal{P}(#1)}

\newcommand{\tdots}{{\cdot}{\cdot}{\cdot}}
\newcommand{\sdots}{{...}}

\let\footnote=\endnote
 
\title[The Hamiltonicity of $\TheCayleyDigraph{{n}}$]%
{Hamiltonicity of the Cayley Digraph on the Symmetric Group Generated by $\sigma = (1 \, 2 \, \tdots \, n)$ and $\tau = (1 \, 2)$\vspace{-0.5em}}

\begin{document}

\begin{abstract}
The symmetric group is generated by $\sigma = (1 \; 2 \; \tdots \; n)$ and $\tau = (1 \; 2)$.
We answer an open problem of Nijenhuis and Wilf by constructing a Hamilton path in the directed Cayley graph for all $n$, and a Hamilton cycle for odd $n$.
\begin{center}
\emph{Dedicated to Herb Wilf (1931 -- 2012).}
\end{center}
\vspace{-1em}
\end{abstract}

\author{Aaron Williams\vspace{-0.5em}}
\thanks{Supported by NSERC Accelerator and Discovery grants, and an ONR basic research grant}

\maketitle
\vspace{-1em}

\section{Introduction} \label{sec:intro}

The Hamiltonicity of Cayley graphs is a well-studied area (see~Pak and Radoi\u{c}i\'{c}'s survey \cite{Pak}).
We consider the \emph{directed $\sigma$-$\tau$ graph} $\TheDigraph{n}=\TheCayleyDigraph{n}$ on the symmetric group $\SymGroup{n}$ with generators $\sigma \;{=}\; (1 \, 2 \, \tdots \, n)$ and $\tau \;{=}\; (1 \, 2)$, and edges $\SigmaEdges \cup \TauEdges$. 

\begin{figure}[h]
\begin{tabular}{@{}cc@{}}
\includegraphics[width=0.485\textwidth]{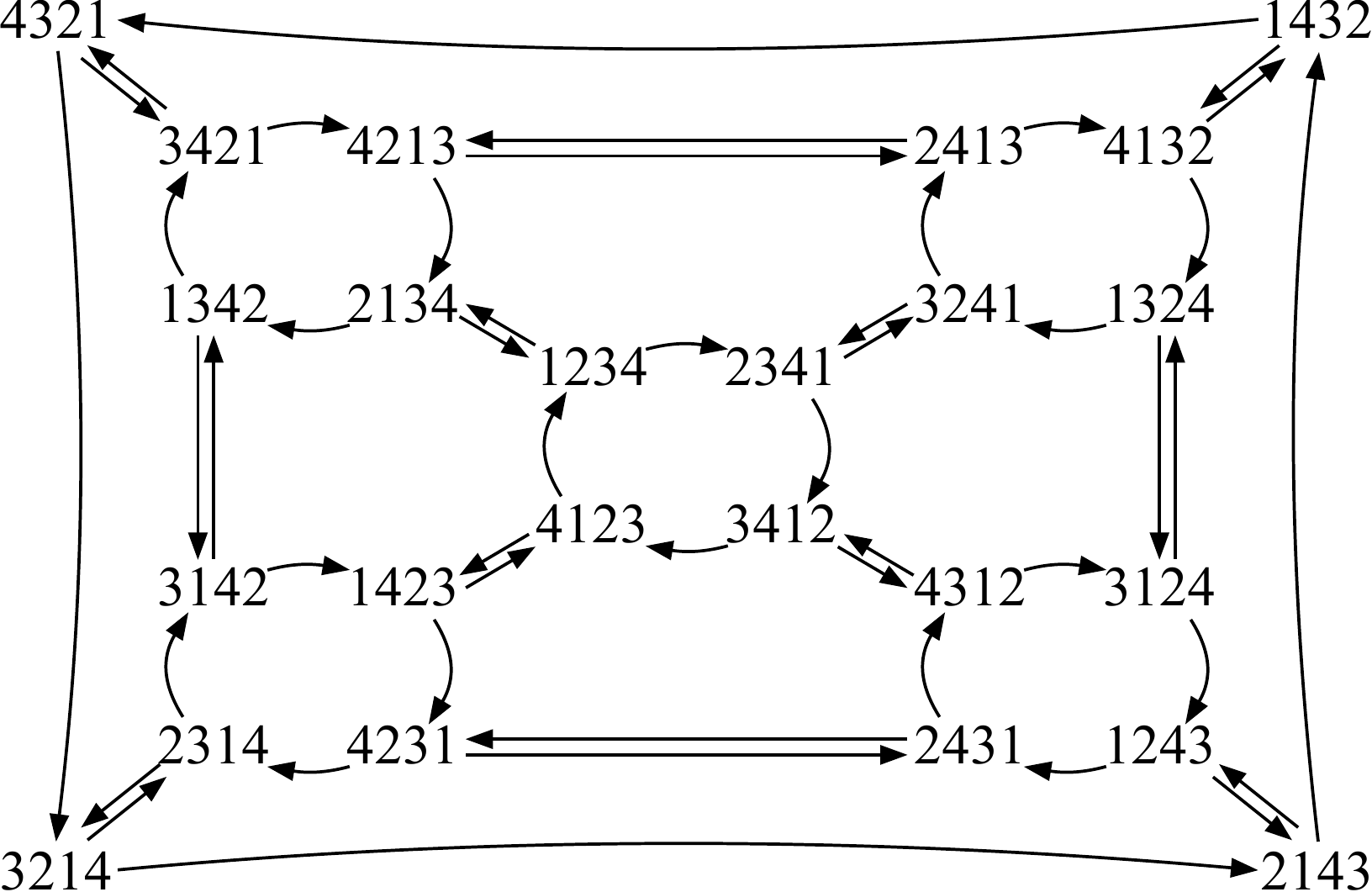} &
\includegraphics[width=0.485\textwidth]{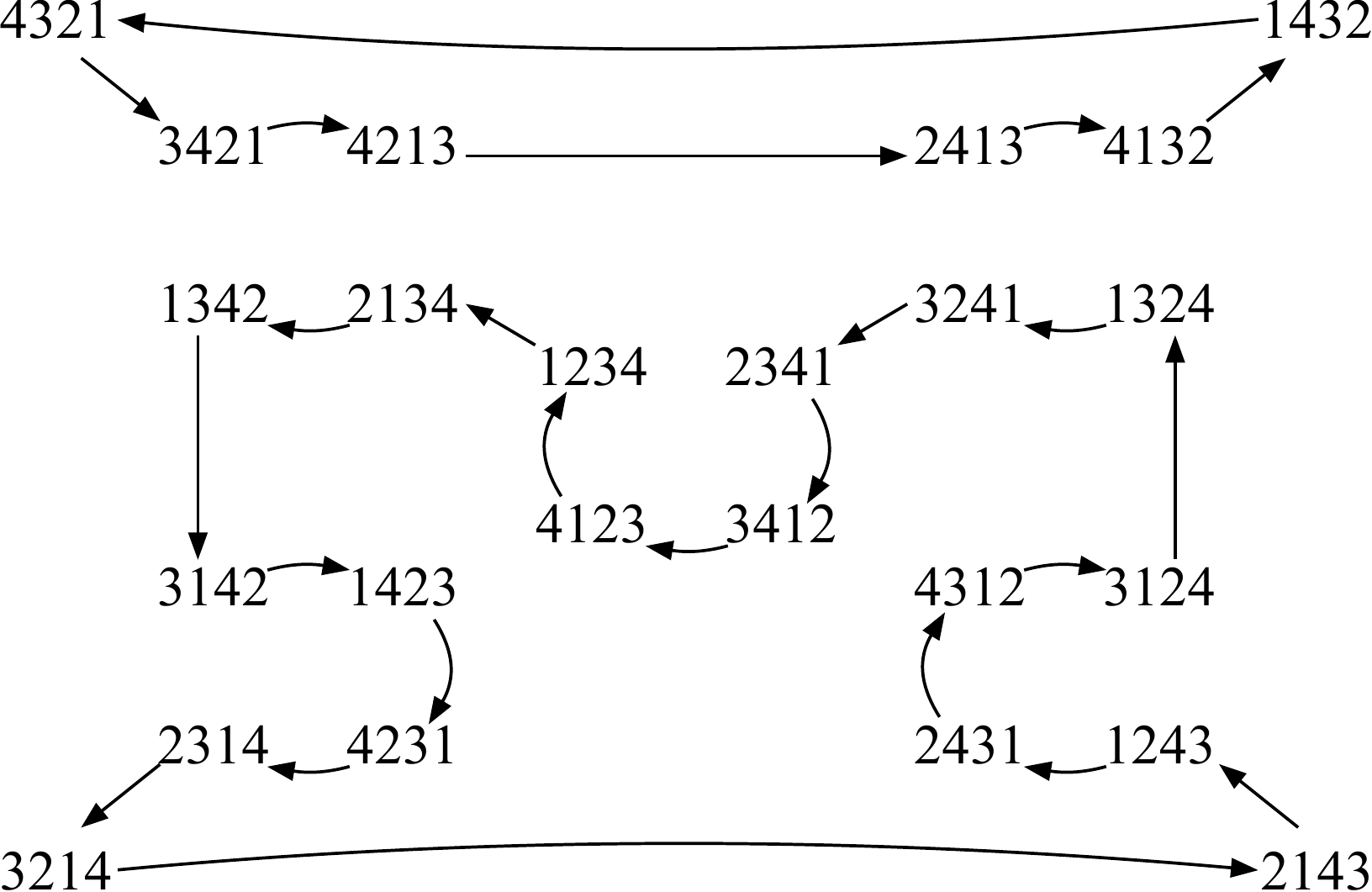} 
\vspace{-0.1in} \\ 
a) $\TheDigraph{4}$ & b) $\TheCover{4}$
\end{tabular}
\vspace{-0.1in}
\caption{a) $\TheDigraph{4}$ with curved and straight arcs for $\SigmaEdges$ and $\TauEdges$, and b) $\TheCover{4}$ is a (disjoint directed) cycle cover of size two in $\TheDigraph{4}$.
}
\label{fig:TheGraphTheCover4}
\end{figure}

The Hamiltonicity of $\TheDigraph{n}$ was first considered by Nijenhuis and Wilf (Exercise 6 in \cite{Wilf}).
A general condition by Rankin \cite{Rankin} forbids Hamilton cycles in $\TheDigraph{n}$ for even $n \geq 4$ (see Swan for a simplified proof \cite{Swan}).
Determining if Hamilton cycles exist for odd $n$ was given a difficulty rating of 48/50 by Knuth, making it one of the hardest in \emph{The Art of Computer Programming} (Problem 71 in Section 7.2.1.2~\cite{Knuth}).
We settle the Hamiltonicity problems by constructing $\TheCycle{n}$ and $\ThePath{n}$ such that:
\begin{enumerate}[noitemsep,nolistsep]
\item[1.] $\TheCycle{n}$ is a Hamilton cycle in $\TheDigraph{n}$ for odd $n$. 
\item[2.] $\ThePath{n}$ is a Hamilton path in $\TheDigraph{n}$ for all $n$.
\end{enumerate}

A \emph{disjoint directed cycle cover} (or simply \emph{cycle cover}) is a set of edges in which each vertex has in-degree one and out-degree one.
A cycle cover partitions into the edges of vertex-disjoint directed cycles that span the vertices, and its \emph{size} is the number of these cycles.
Our Hamilton cycle $\TheCycle{n}$ is a cycle cover of size one, and our Hamilton path $\ThePath{n}$ splits and joins the two cycles in the following:
\begin{enumerate}[noitemsep,nolistsep]
\item[3.] $\TheCover{n}$ is a cycle cover of size two in $\TheDigraph{n}$ for all $n$.
\end{enumerate}
Both $\TheCycle{n}$ and $\TheCover{n}$ are created by specifying the edge that enters each vertex $\mathbf{p}$.
%Both $\TheCycle{n}$ and $\TheCover{n}$ are created by incoming edge rules for each vertex $\mathbf{p}$.

\begin{definition} \label{def:TheCycle}
Let $\mathbf{p} \,{=}\, p_0 \tdots p_{n-1}$, $p_i \,{=}\, n$, and $r \,{=}\, p_{(i \bmod n{-}1){+}1}$. 
Then $(\mathbf{p} \tau, \mathbf{p}) \,{\in}\, \TheCycle{n}$ if (a) $r \,{<}\, n{-}1$ and $r \,{=}\, p_0{-}1$, or (b) $r\,{=}\,n{-}1$ and $p_0 \,{=}\, 2$, or (c) $p_1 p_2 \tdots p_{n-1}$ is a rotation of $1 \, 2 \, \tdots \, n{-}1$; 
otherwise, $(\mathbf{p} \sigma^{-1}, \mathbf{p}) \,{\in}\, \TheCycle{n}$.
\end{definition} 

\begin{definition} \label{def:TheCover}
Let $\mathbf{p} \,{=}\, p_0 \tdots p_{n-1}$, $p_i \,{=}\, n$, and $r \,{=}\, p_{(i \bmod n{-}1){+}1}$. 
Then $(\mathbf{p} \tau, \mathbf{p}) \,{\in}\, \TheCover{n}$ if (a) $r \,{<}\, n{-}1$ and $r \,{=}\, p_0{-}1$, or (b) $r\,{=}\,n{-}1$ and $p_0 \,{=}\, 1$; 
otherwise, $(\mathbf{p} \sigma^{-1}, \mathbf{p}) \,{\in}\, \TheCover{n}$.
\end{definition} 

Let us decipher Definition \ref{def:TheCover}.
The symbol $n$ has index $i$ in $\mathbf{p}$, and $r$ denotes the symbol to the right of $n$, except that $r = p_1$ when $i = n{-}1$.
For example, if $\mathbf{p} = p_0p_1p_2p_3 = 4231$, then $i=0$ and $r = 2$ (since $p_0=4$ is followed by $p_1 = 2$).
If $\mathbf{p} = 1234$, then $i=3$ and $r=2$ by the exceptional case.
A $\tau$-edge enters $\mathbf{p}$ when its first symbol is $p_0 = (r \bmod n{-}1){+}1$.
Thus, the vertices entered by a $\tau$-edge in $\TheCover{4}$ are 
$24\underline{1}3, 234\underline{1}, 2\underline{1}34$, and
$34\underline{2}1, 314\underline{2}, 3\underline{2}14$, and
$14\underline{3}2, 124\underline{3}, 1\underline{3}24$, where $r$ is underlined.
Figure \ref{fig:TheGraphTheCover4} b) gives $\TheCover{4}$ by adding $\sigma$-edges into the other vertices.

Although our rules are succinct, they do not lead directly to proofs of our results.
Instead we express $\TheCycle{n}$ and $\TheCover{n}$ as the symmetric difference of cycles of the form $\sigma^n$ and $(\tau \sigma^{-1})^{n-1}$.
Our main lemma computes cycle cover sizes in $\TheDigraph{n}$ using rotation systems, and we spend considerable effort simplifying these systems. In particular, our Hamilton cycle simplifies to a wheel structure on $n{-}1$ vertices.

The technique of creating large cycles from the symmetric difference of small cycles has been used by change ringers for hundreds of years (see Duckworth and Stedman~\cite{Tintinnalogia}).
Prior to this article, Hamilton cycles of the undirected $\TheCayleyGraph{n}$ were constructed with great difficultly by Compton and Williamson~\cite{Compton}. 
Hamilton cycles in $\CayleyG{n}$ have also been constructed for $G = \{\sigma, \tau, (1 \, 2 \, 3)\}$ by Stevens and Williams \cite{Stevens}, and $G = \{\sigma, (1 \, 2 \, \tdots \, n{-}1)\}$ by Holroyd, Ruskey, and Williams \cite{Holroyd}.
The literature frequently states that $\TheDigraph{5}$ is not Hamiltonian; see Ruskey, Jiang, and Weston \cite{Ruskey} for the history and resolution of this error.
Our local rules translate into efficient algorithms for generating permutations, and so this article has applications to combinatorial Gray codes (see Savage \cite{Savage} and Knuth~\cite{Knuth}). 

Sections \ref{sec:preliminaries}--\ref{sec:Hamilton} cover preliminaries, cycles in $\TheDigraph{n}$, rotation systems, and our theorems, respectively.
Table \ref{tab:summary} gives a handy summary of our notation on page~\pageref{tab:summary}.

\section{Preliminaries} \label{sec:preliminaries}

This section collects background concepts and definitions. %, conventions, and terminology.
In this article we often use subscripts and superscripts cyclically, so by convention $x_{k+1} = x_1$ when discussing a set $\{x_1, x_2, \sdots, x_k\}$, and $x^0 = x_k$ when discussing a sequence $x^1 \; x^2 \; \tdots \; x^k$.

\subsection{Strings} \label{sec:preliminaries_strings}

The set of permutations of $\nset{n} = \{1,2,\ldots,n\}$ written as strings is
\begin{equation*} 
\PERMS{n} = \big\{ p_1 \tdots p_n : \{p_1,\sdots,p_n\} = \nset{n}\big\}.
\end{equation*}
The strings that are \emph{missing} $m \in \nset{n}$, or any one symbol, are defined~as
\begin{equation*}
\MISSING[m]{n} = \big\{ p_1 \tdots p_{n-1} : \{p_1,\sdots,p_{n-1}\} = \nset{n} \backslash \{m\}\big\} 
\text{ and }
\MISSING{n} = \MISSING[1]{n} \cup \MISSING[2]{n} \cup \tdots \MISSING[n]{n}.
\end{equation*}
Thus, $\PERMS{3} \;{=}\; \{123,132,213,231,312,321\}$ contains the elements of $\SymGroup{3}$ in one-line notation, $\MISSING[1]{3} = \{23,32\}$ contains strings missing $1$, and $\MISSING{3} = \{12,13,21,23,31,32\}$ contains strings missing any one symbol.
We use bold letters for strings and subscripts for their individual symbols, as in $\mathbf{p} = p_1 p_2 \tdots p_n \in \PERMS{n}$.
We let $\mathbf{p} \sigma$ and $\mathbf{p} \tau$ denote the application of $\sigma$ and $\tau$ to the indices of $\mathbf{p} = p_1 p_2 \tdots p_n \in \PERMS{n}$.
Thus, $\mathbf{p} \sigma = p_2 p_3 \tdots p_n p_1$ since the first symbol is $p_2$, the second is $p_3$, and so on.
We apply multiple $\sigma$ and $\tau$ from left-to-right, and use exponentiation for inverses and repetition.
For example, $\mathbf{p} \sigma^{-1} = p_n p_1 p_2 \tdots p_{n-1}$ and $\mathbf{p} (\tau \sigma)^2 = \mathbf{p} \tau \sigma \tau \sigma = ((((\mathbf{p} \tau) \sigma) \tau ) \sigma)$.

\subsection{Rotational Equivalence} \label{sec:preliminaries_equiv}

Let $\sim$ denote the equivalence relation for strings under rotation.
We use uppercase and bold uppercase for equivalence classes and sets of equivalence classes, as in $X = [312] = \{312, 123, 213\}$, and $\mathbf{Y} = \{[12], [23]\}$.
Quotient sets are specified by $\quotient{}$.
Thus, $\quotient{\PERMS{3}} = \{[321], [312]\}$, $\quotient{\MISSING[1]{3}} = \{[32]\}$, and $\quotient{\MISSING{3}} = \{[31], [32], [21]\}$.
We say $X \in \quotient{\PERMS{n}}$ and $Y \in \quotient{\MISSING{n}}$ are \emph{consistent} if $X = [p_1 p_2 \tdots p_n]$ and $Y = [p_2 p_3 \tdots p_n]$.
In other words, $X$ and~$Y$ are consistent if they have the same circular order when ignoring the missing symbol.
For example, $X = [4123]$ and $Y = [134]$ are consistent (by $p_1p_2p_3p_4 = 2341$).

\subsection{Walks, Paths, and Cycles} \label{sec:preliminaries_walks}

Let $G=(V,E)$.
A \emph{walk} of \emph{length $k$} is a sequence $v_1 \ e_1 \ v_2 \ e_2 \ \sdots \ v_k \ e_k \ v_{k+1}$ where $e_i \in E$ is incident with $v_i, v_{i+1} \in V$ for $i \in \nset{k}$.
A \emph{path} is a walk with distinct edges and vertices.
A \emph{cycle} is a path except that $v_{k+1} = v_1$, and we omit $v_{k+1}$ from the sequence.
If $G$ is directed, then $e_i = (v_i, v_{i+1})$ is a \emph{forward edge} and $e_i = (v_{i+1}, v_i)$ is a \emph{backward edge} on the walk.
A path or cycle is \emph{directed} if it has no backward edges.
If $e_i = (v_i, \alpha v_i)$ for $\alpha \in \{\sigma,\sigma^{-1},\tau\}$, then we can replace $e_i$ by $\alpha$ in the sequence.
We specify walks and their edge sets~by 
\begin{gather*}
\walk{\mathbf{p} \alpha_1 \alpha_2 \tdots \alpha_k} = \mathbf{p} \ \alpha_1 \ (\mathbf{p}\alpha_1) \ \alpha_2 \ (\mathbf{p}\alpha_1\alpha_2) \ \alpha_3 \ \tdots \ \alpha_k \ (\mathbf{p} \alpha_1 \alpha_2 \tdots \alpha_k) \text{ and } \\
\edges{\mathbf{p} \alpha_1 \alpha_2 \tdots \alpha_k} = \{
(\mathbf{p}, \mathbf{p}\alpha_1), \ 
(\mathbf{p}\alpha_1, \mathbf{p}\alpha_1\alpha_2), \ 
\sdots, \ 
(\mathbf{p} \alpha_1 \alpha_2 \tdots \alpha_{k-1}, \mathbf{p} \alpha_1 \alpha_2 \tdots \alpha_k)\}.
\end{gather*}
For example, one of the cycles in Figure \ref{fig:TheGraphTheCover4} b) is $\walk{\mathbf{p} (\tau \sigma)^3}$ for $\mathbf{p} = 4321$.

\section{The Directed $\sigma$-$\tau$ Graph and its Cycles} \label{sec:structure}

This section shows how cycle covers can be expressed as the symmetric difference of basic cycles.
Figure \ref{fig:Cycles4} illustrates our reformulation of $\TheCover{4}$ from Figure~\ref{fig:TheGraphTheCover4}~b).

\subsection{Basic Cycles} \label{sec:structure_basic}

We first show that walks associated with $\sigma^n$ and $(\tau \sigma^{-1})^{n-1}$ are (directed) cycles, which we refer to as \emph{$\sigma$-cycles} and \emph{alternating-cycles}, respectively.

\begin{lemma} \label{lem:cycles}
If $\mathbf{p} = p_1 \tdots p_n \in \PERMS{n}$, then $\walk{\mathbf{p} \sigma^n}$ and $\walk{\mathbf{p} (\tau \sigma^{-1})^{n-1}}$ are cycles. 
\end{lemma}
\begin{proof}
The sequences give cycles on multiple lines below left and right, respectively.
\begin{equation*}
\begin{array}{@{}c@{\,}c@{\,}c@{\,}c@{\,}c@{\,}c@{ \ \ }c}
p_1 & p_2 & p_3 & p_4 & \tdots & p_n & \sigma \\
p_2 & p_3 & p_4 & \tdots & p_n & p_1 & \sigma \\
\multicolumn{5}{c}{\cdots} & & \cdots \\
p_n & p_1 & p_2 & p_3 & \tdots & p_{n-1} & \sigma
\end{array}
\ \vline \ 
\begin{array}{*{2}{c@{\,}c@{\,}c@{\,}c@{\,}c@{\,}c@{\,}c@{ \ \ }c}@{}}
p_1 & p_2 & p_3 & p_4 & p_5 & \tdots & p_n & \tau & p_2 & p_1 & p_3 & p_4 & p_5 & \tdots & p_n & \sigma^{-1} \\
p_n & p_2 & p_1 & p_3 & p_4 & \tdots & p_{n-1} & \tau & p_2 & p_n & p_1 & p_3 & p_4 & \tdots & p_{n-1} & \sigma^{-1} \\ 
\multicolumn{7}{c}{\ldots} & \ldots & \multicolumn{7}{c}{\ldots} & \ldots \\
p_3 & p_2 & p_4 & p_5 & \tdots & p_n & p_1 & \tau & p_2 & p_3 & p_4 & p_5 & \tdots & p_n & p_1 & \sigma^{-1}
\end{array}
\end{equation*}
Notice that $\walk{\mathbf{p} \sigma^{n}}$ visits $[\mathbf{p}]$, while $\walk{\mathbf{p} (\tau \sigma^{-1})^{n-1}}$ visits $[p_1 p_3 p_4 \tdots p_n]$ with each $\tau$-edge moving the missing symbol $p_2$ from the second to first position.
\end{proof}

Let $\SigmaCycles{n} \subseteq 2^{\SigmaEdges}$ and $\AltCycles{n} \subseteq 2^{\SigmaEdges \cup \TauEdges}$ contain the edge sets of each $\sigma$-cycle and alternating-cycle in $\TheDigraph{n}$, respectively.
We form bijections between these sets and our equivalence classes based on the following equalities,
\begin{equation*}
|\SigmaCycles{n}| = n! / n = (n{-}1)! = |\quotient{\PERMS{n}}| \text{ and } 
|\AltCycles{n}| = 2n! / (2n{-}2) = n(n{-}2)! = |\quotient{\MISSING{n}}|
\end{equation*}
which are due to each vertex being on one $\sigma$-cycle (of length $n$) and two alternating-cycles (of length $2n{-}2$)
We define our one-to-one maps $s : \quotient{\PERMS{n}} \rightarrow \SigmaCycles{n}$ and $a : \quotient{\MISSING{n}} \rightarrow \AltCycles{n}$ below, and we note that they respect $\sim$ by the proof of Lemma~\ref{lem:cycles}
\begin{align*}
s([\mathbf{p}]) &= \edges{\mathbf{p} \sigma^n} \text{ for } \mathbf{p} \in \PERMS{n} \\
a([\mathbf{q}]) &= \edges{\mathbf{p} (\tau \sigma^{-1})^{n-1}} \text{ for } \mathbf{q} = q_1 q_2 \tdots q_{n-1} \in \MISSING[m]{n} \text{ and } \mathbf{p} = q_1 m q_2 q_3 \tdots q_{n-1} \in \PERMS{n}.
\end{align*}
For example, $s([4123]) = \edges{4123 \, \sigma^4}$ and $a([413]) = \edges{4213 \, (\tau \sigma^{-1})^{3}}$.

\subsection{Cycle Covers} \label{sec:structure_covers}

Now we express cycle covers as symmetric differences. % of basic cycles.

\begin{lemma} \label{lem:xor} 
$D$ is a cycle cover in $\TheDigraph{n}$ if and only if $D=E_\sigma \xor (A_1 \cup A_2 \cup \tdots \cup A_h)$ for a subset of alternating-cycle edge sets $\{A_1, A_2, \sdots, A_h\} \subseteq \AltCycles{n}$.
\end{lemma}

\begin{proof}
Suppose $D = E_\sigma \xor (A_1 \cup A_2 \cup \tdots \cup A_h)$.
Note that $E_\sigma$ is a cycle cover. 
Distinct alternating-cycles are edge-disjoint, and vertex $\mathbf{p}$ on an alternating-cycle~$A$ either has one $\sigma$-edge and one $\tau$-edge entering it on $A$, or one $\sigma$-edge and one $\tau$-edge exiting it in $A$.
Thus, $\mathbf{p}$ has in-degree and out-degree one in $D$, so $D$ is a cycle~cover.

Suppose $D$ is a cycle cover.
If the $\tau$-edge $(\mathbf{p}, \mathbf{p} \tau)$ is in $D$, then the $\sigma$-edge $(\mathbf{p} \tau \sigma^{-1}, \mathbf{p} \tau)$ is not in $D$, and hence the $\tau$-edge $(\mathbf{p} \tau \sigma^{-1}, \mathbf{p} \tau \sigma^{-1} \tau)$ is in $D$, and so on.
Thus, $D$ contains every $\tau$-edge and none of the $\sigma$-edges on an alternating-cycle.
Since each cycle in $D$ either contains a $\tau$-edge or is a $\sigma$-cycle, we can conclude that $D = E_\sigma \xor (A_1 \cup A_2 \cup \tdots \cup A_h)$ for some $\{A_1, A_2, \sdots, A_h\} \subseteq \AltCycles{n}$.
\end{proof}

Now we apply Lemma \ref{lem:xor} to $\TheCycle{n}$ and $\TheCover{n}$.
To describe the alternating-cycles, let $\FixSubset[n]{r}{m} = \{[n r p_3 \tdots p_{n-1}] \in \quotient{\MISSING[m]{n}}\}$ contain equivalence classes with two \emph{fixed symbols}: $r$ is right of $n$, and $m$ is missing.
For example, $\FixSubset[5]{1}{2} = \{[5134],[5143]\}$. 
We let $\FixSubset{r}{m}=\FixSubset[n]{r}{m}$ when context allows, and combine these subsets as follows
\begin{align}
\YCycle{n} &= \FixSubset{1}{2} \cup \FixSubset{2}{3} \cup \tdots \cup \FixSubset{n{-}2}{n{-}1} \cup \FixSubset{n{-}1}{2} \cup \{[1 \, 2 \, \tdots \, n{-}1]\} \text{ and} \label{eq:YCycle} \\
\YCover{n} &= \FixSubset{1}{2} \cup \FixSubset{2}{3} \cup \tdots \cup \FixSubset{n{-}2}{n{-}1} \cup \FixSubset{n{-}1}{1}.  \label{eq:YCover} 
\end{align}
For example, the two sets are below for $n=5$ along with their fixed symbol subsets 
\begin{align*}
\YCycle{5} 
&= \{\overset{\FixSubset{1}{2}}{[5143],[5134]},\overset{\FixSubset{2}{3}}{[5241],[5214]},\overset{\FixSubset{3}{4}}{[5321],[5312]},\overset{\FixSubset{4}{2}}{[5431],[5413]},[1234]\} \\
\YCover{5} 
&= \{\underset{\FixSubset{1}{2}}{[5143],[5134]},\underset{\FixSubset{2}{3}}{[5241],[5214]},\underset{\FixSubset{3}{4}}{[5321],[5312]},\underset{\FixSubset{4}{1}}{[5432],[5423]}\}.
\end{align*}
Let $\ACycle{n}$ and $\ACover{n}$ be the respective unions of $a(Y)$ for $Y \in \YCycle{n}$ and $Y \in \YCover{n}$.
For example, $\ACover{4} = a([413]) \cup a([421]) \cup a([432])$ appears in Figure \ref{fig:Cycles4} b).

\begin{lemma} \label{lem:rules}
(i) $\TheCycle{n} = \SigmaEdges \xor \ACycle{n}$ and (ii) $\TheCover{n} = \SigmaEdges \xor \ACover{n}$.
\end{lemma}
\begin{proof}
Consider $\mathbf{p} = p_0 \tdots p_{n-1}$ with respect to Definition \ref{def:TheCover}.
Notice that 
\begin{center}
(a) $r < n{-}1$ and $r = p_0{-}1$, or (b) $r = n{-}1$ and $p_0 = 1$ $\iff$ $\FixSubset{r}{p_0} \subseteq \YCover{n}$.
\end{center}
Thus, $(\mathbf{p} \tau, \mathbf{p}) \in \TheCover{n}$ if and only if $(\mathbf{p} \tau, \mathbf{p}) \in \ACover{n}$, and $(\mathbf{p} \sigma^{-1}, \mathbf{p}) \in \TheCover{n}$ if and only if $(\mathbf{p} \sigma^{-1}, \mathbf{p}) \notin \ACover{n}$.
This proves (ii) since $\SigmaEdges \xor \ACover{n}$ contains every $\tau$-edge in $\ACover{n}$ and every $\sigma$-edge not in $\ACover{n}$.
A similar argument proves (i).
\end{proof}

\begin{figure}[h]
\begin{tabular}{@{}cc@{}}
\includegraphics[width=0.475\textwidth]{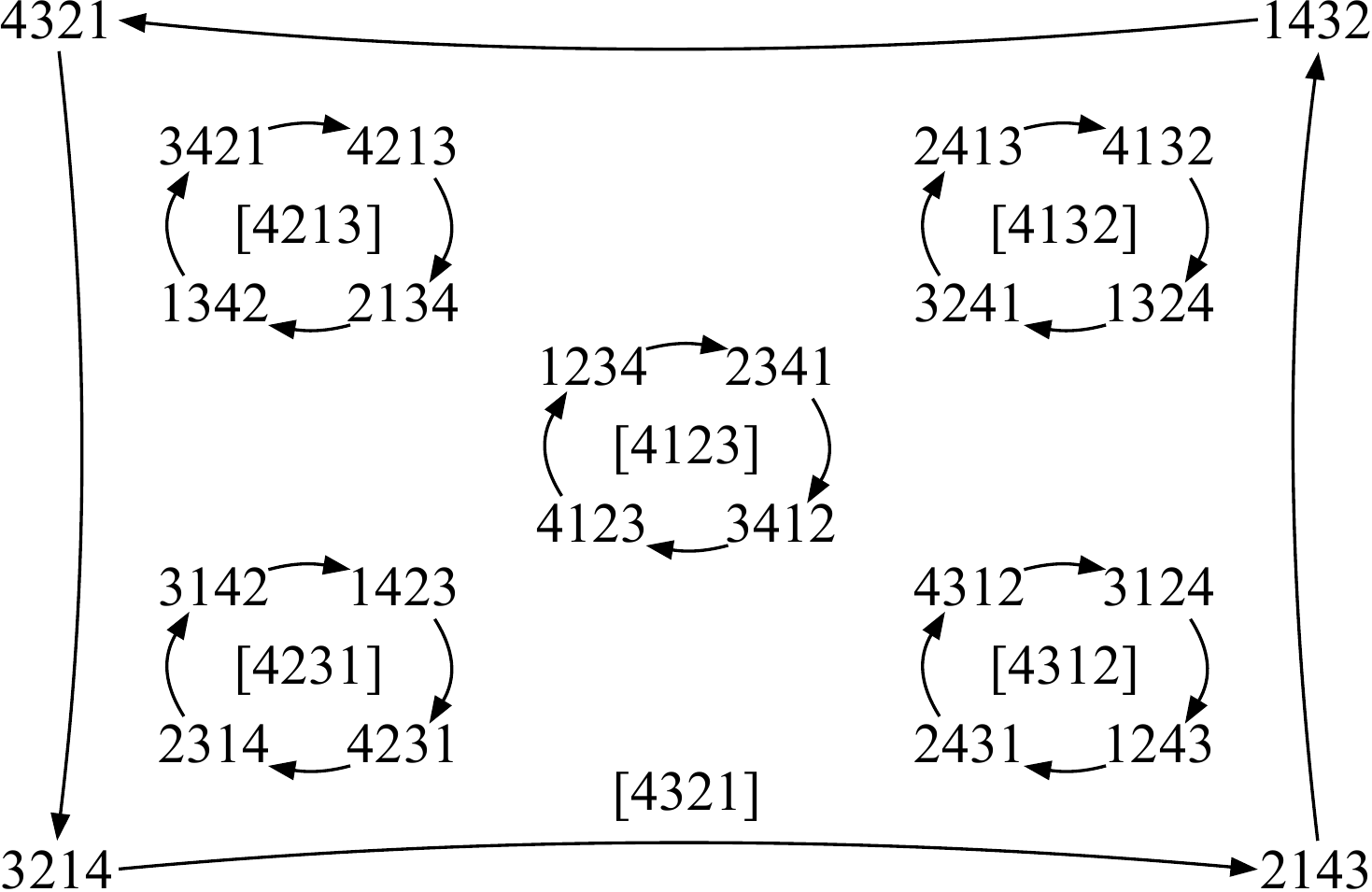} & 
\includegraphics[width=0.475\textwidth]{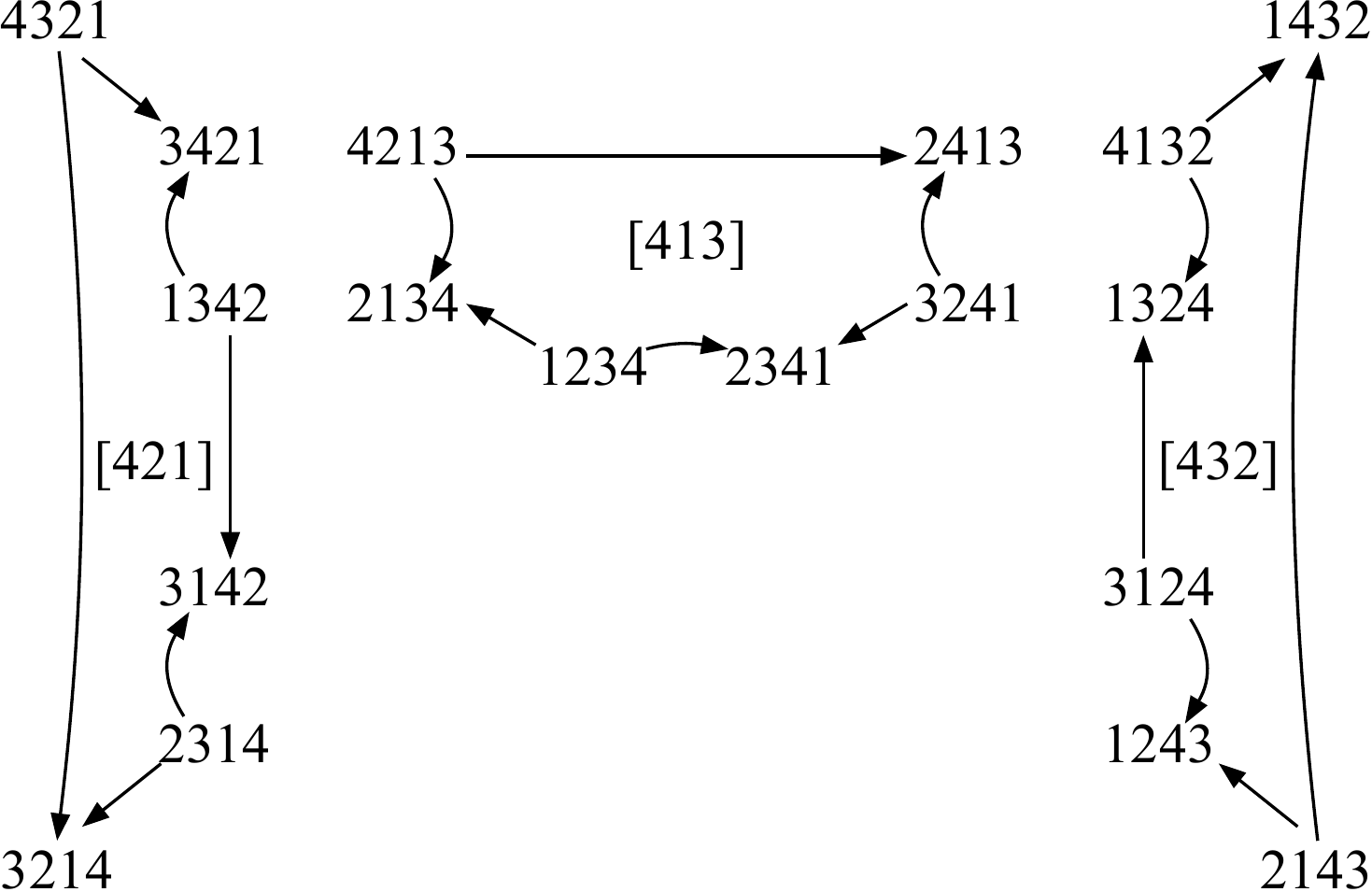}
\vspace{-0.125in}
\\
a) $\SigmaEdges$ & b) $\ACover{4}$
\end{tabular}
\vspace{-0.125in}
\caption{
The symmetric difference of a) and b) is the cycle cover $\TheCover{4}$. 
}
\label{fig:Cycles4}
\end{figure}

\section{Rotation Systems} \label{sec:rotation}

This section discusses rotation systems in general, and two systems in particular. 
We name $\Wheeel{n}$ after the wheel graph, and show that it has one or two faces depending on $n$'s parity.
We name $\Wilf{n}$ after Wilf \cite{Wilf}, and use it to compute cycle cover sizes. 

\subsection{Definitions} \label{sec:rotation_basic}

A \emph{rotation system} $R=(V,E,\theta)$ (or \emph{combinatorial embedding})  is an undirected graph $G=(V,E)$ that allows loops and parallel edges, and a cyclic order $\theta(v)$ on the edges incident with each $v \in V$.
A \emph{face} is a cycle $v_1 \ e_1 \ v_2 \ e_2 \ \tdots \ v_k \ e_k$ in which each $e_{i-1}$ is immediately followed by $e_{i}$ in $\theta(v_i)$. 
Each edge $(u,v) \in E$ is comprised of two \emph{darts}, $\dart{u}{v}$ and $\dart{v}{u}$, and each dart belongs to one face.
Rotation systems are \emph{unifacial} (also known as \emph{unicellular}) or \emph{bifacial} if they have one or two faces,~respectively.
To simplify our figures and formulae, we use \CW\ and \CCW\ for vertices that are embedded with clockwise and counter-clockwise edge orders, respectively, and in writing we specify $\theta(v)$ with any linear order that induces it.

\subsection{Spinning `Wheeel'} \label{sec:rotation_Wheeel}

An \emph{$m$-wheel} is an undirected graph with vertex set $\{h, r_1, \sdots, r_{m-1}\}$ and edges $R^j = (r_{j}, r_{j+1})$ and $S^j = (h, r_j)$ for $j \in \nset{m}$. 
The cycle $r_1 \ R^1 \ r_2 \ R^2 \ r_3 \ \tdots \ r_{m-1} \ R^{m-1}$ is the \emph{rim}, $h$ is the \emph{hub}, and the $S^j$ edges are the \emph{spokes}.
An \emph{$m$-wheeel} (with an extra `\emph{e}') is an $m$-wheel plus an edge $P^1 = (h, r_{1})$ parallel to $S^1$.
The \emph{spinning $m$-wheeel} is a rotation system $\Wheeel{m} = (V,E,\theta)$ where $(V,E)$ is an $m$-wheeel and $\theta$ orders the edges as follows for $2 \leq i \leq n{-}1$
\begin{equation*}
\theta(r_1) = S^1, R^1, P^1, R^{m-1} \ \ \ 
\theta(r_i) = S^i, R^i, R^{i-1} \ \ \ 
\theta(h) = S^1, S^2, \sdots, S^{m-2}, P^1, S^{m-1}.
\end{equation*}
Figure \ref{fig:Wheeel6} illustrates $\Wheeel{7}$ and $\Wheeel{6}$ and the proof of Lemma \ref{lem:Wheeel}.

\begin{figure}[h]
\begin{tabular}{@{}c@{}c@{\;}c@{\;\;\;}c@{}}
\includegraphics[width=0.24\textwidth]{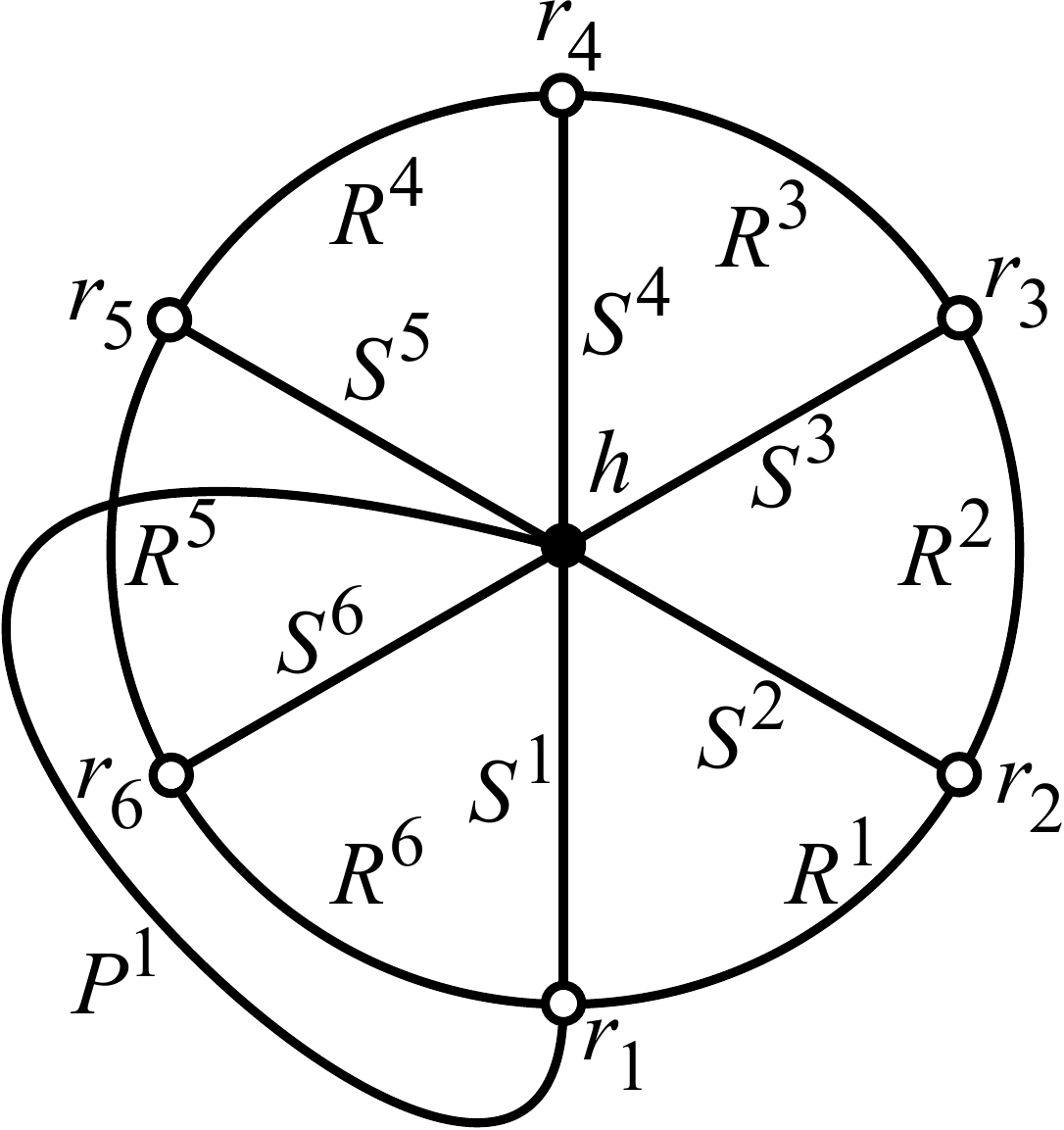} &
\includegraphics[width=0.24\textwidth]{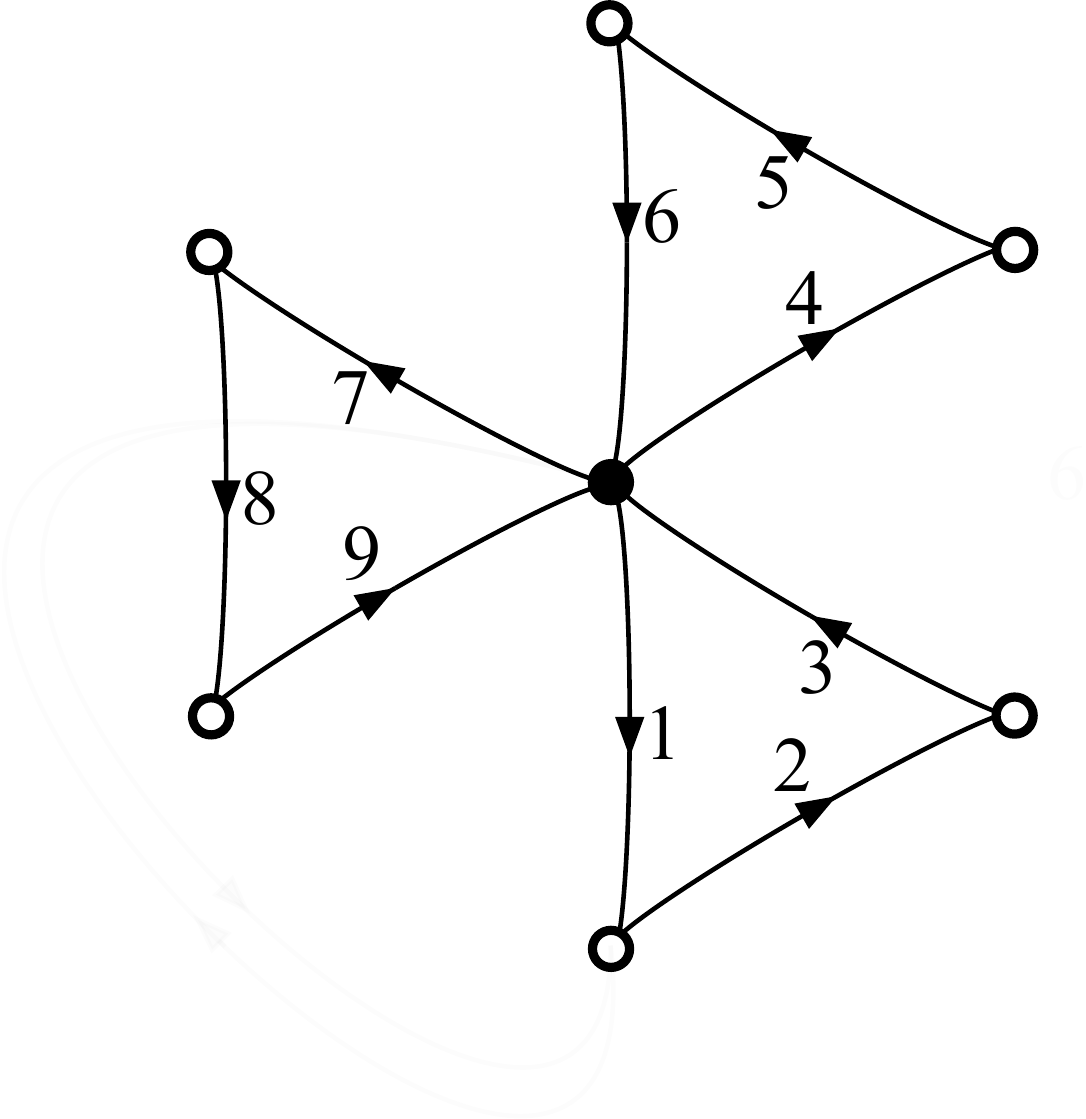} &
\includegraphics[width=0.24\textwidth]{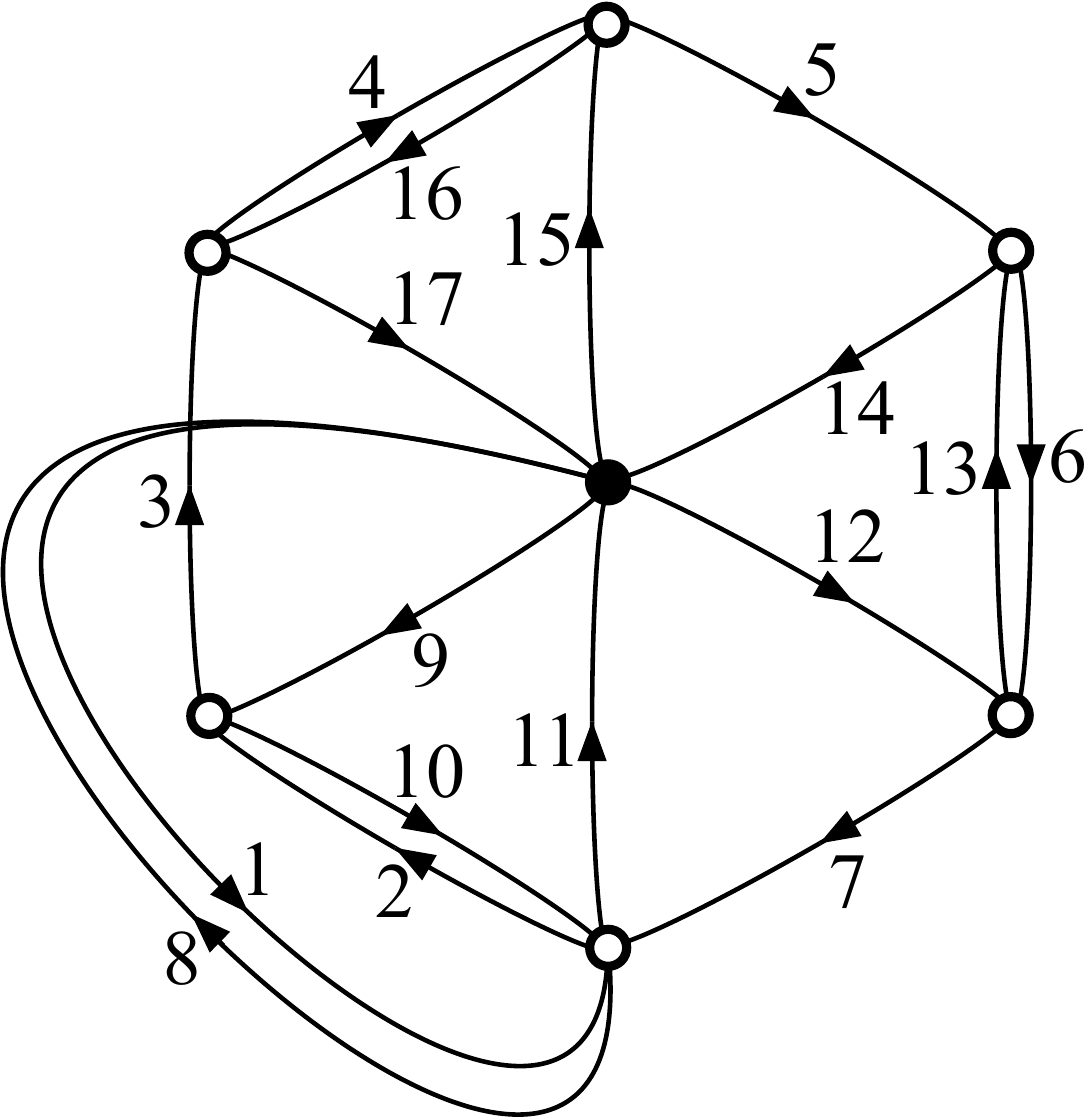} &
\includegraphics[width=0.24\textwidth]{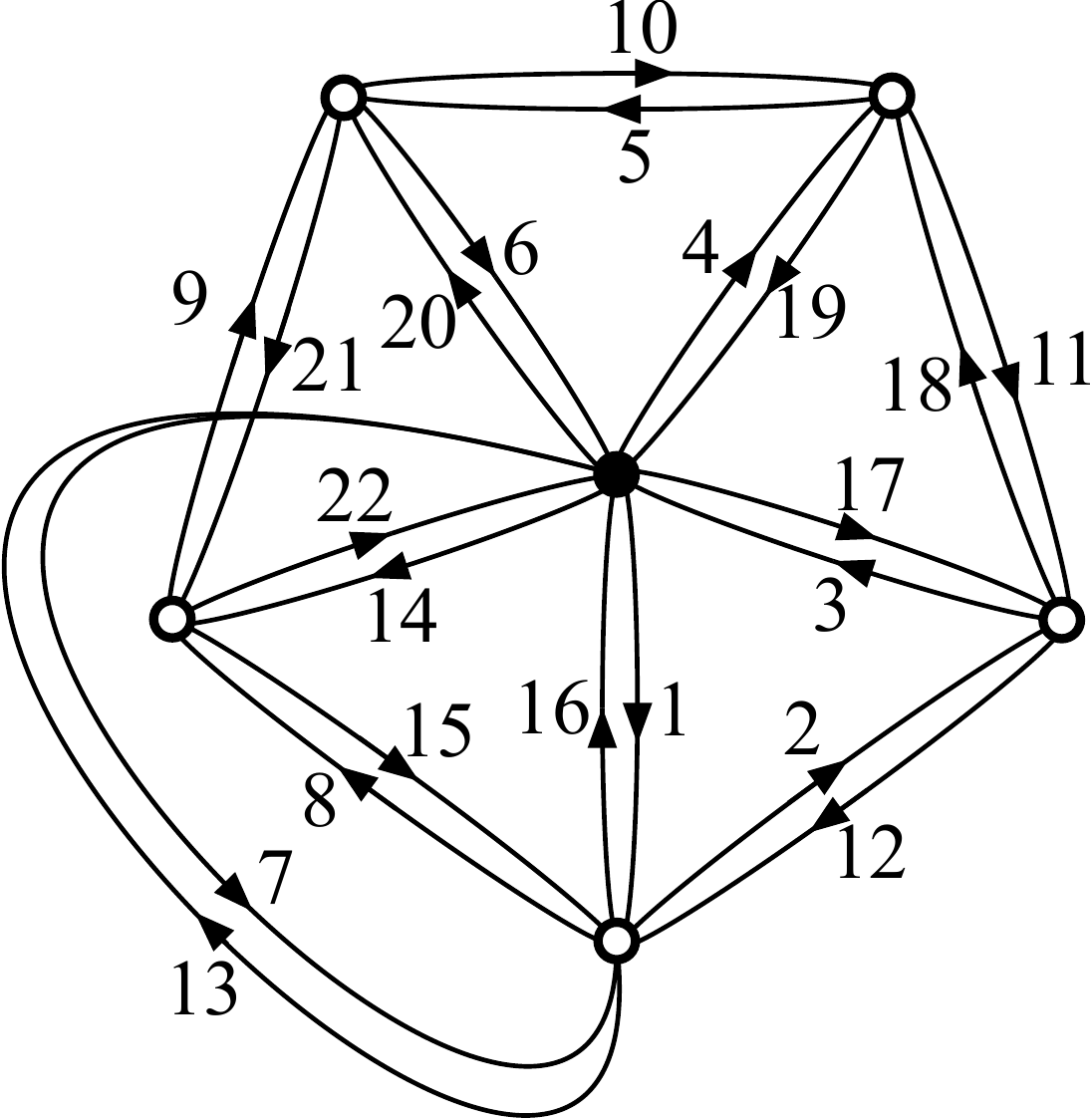} \\
\vspace{-0.355in} \\ 
\hspace{0.5in} a) &
\multicolumn{2}{c}{b)} &
\hspace{0.5in} c)
\end{tabular}
\vspace{-0.1in}
\caption{a) $\Wheeel{7}$ with edges ordered counter-clockwise for the hub vertex \protect\CCW\ and clockwise for the rim vertices \protect\CW,
b) darts along the two faces of $\Wheeel{7}$,
c) darts along the one face of $\Wheeel{6}$.}
\label{fig:Wheeel6}
\end{figure}

\begin{lemma} \label{lem:Wheeel}
$\Wheeel{m}$ is unifacial if $m$ is even, and is bifacial if $m$ is odd.
\end{lemma}
\begin{proof}
Let $f = \lfloor m/2 \rfloor$ and $c = \lceil m/2 \rceil$. 
$\Wheeel{m}$ has two faces for odd~$m$,
\begin{gather*}
h \ S^{1}  \ r_1  \ R^{1}  \ r_2  \ S^{2}  \ h  \ S^{3}  \ r_3  \ R^{3}  \ r_4  \ \tdots  \ r_{2c-3}  \ R^{2c-3}  \ r_{2c-2}  \ S^{2c-2} \ \text{ and } \\
h  \ P^{1}  \ r_1  \ R^{m-1}  \ r_{m-1}  \ R^{m-2}  \ r_{m-2}  \ R^{m-3}  \ r_{m-3}  \ \tdots  \ r_2  \ R^{1}  \ r_1  \ P^{1} \\
h  \ S^{m-1}  \ r_{m-1}  \ R^{m-1}  \ r_{1}  \ S^{1}  \ h  \ S^{2}  \ r_{2}  \ R^{2}  \ r_{3}  \ S^{3}  \ h  \ S^{4}  \ r_{4}  \ \tdots  \ r_{2f-2}  \ R^{2f-2} \ r_{2f-1} \ S^{2f-1}.
\end{gather*}
When $m$ is even, $\Wheeel{m}$ has one face obtained by removing the `and' above.
\end{proof}

\subsection{Wilf's Rotation System} \label{sec:rotation_Wilf}

\emph{Wilf's graph} is an undirected bipartite graph with vertices $\quotient{\PERMS{n}} \cup \quotient{\MISSING{n}}$ and edges between consistent vertices.
\emph{Wilf's rotation system} is $\Wilf{n} = (V,E,\theta)$ on Wilf's graph $(V,E)$ with the following edge orders for $[\mathbf{p}] = [p_1 p_2 \tdots p_n] \in \quotient{\PERMS{n}}$ and $[\mathbf{q}] = [q_1 q_2 \tdots q_{n-1}] \in \quotient{\MISSING[m]{n}}$
\begin{align}
\theta([\mathbf{p}]) &= 
(\mathbf{p}, [p_2 p_3 \tdots p_n]), \, 
(\mathbf{p}, [p_1 p_3 p_4 \tdots p_n]), \, 
\sdots, \, 
(\mathbf{p}, [p_1 p_2 \tdots p_{n-1}])  \label{eq:thetap} \\
\theta([\mathbf{q}]) &= 
(\mathbf{p}, [q_1 q_2 \tdots q_{n-1} m]), \, 
(\mathbf{p}, [q_1 q_2 \tdots q_{n-2} m q_{n-1}]), \, 
\sdots, \, 
(\mathbf{p}, [q_1 m q_2 q_3 \tdots q_{n-1}]). \label{eq:thetaq}
\end{align}
In other words, edges are ordered by left-to-right deletions for vertices in $\quotient{\PERMS{n}}$ and by right-to-left insertions for vertices in $\quotient{\MISSING{n}}$.
Figure \ref{fig:Wilf4} a) illustrates $\Wilf{4}$. 

\begin{figure}[h]
\begin{tabular}{@{}cc@{}}
\raisebox{-\height}{\includegraphics[scale=0.25]{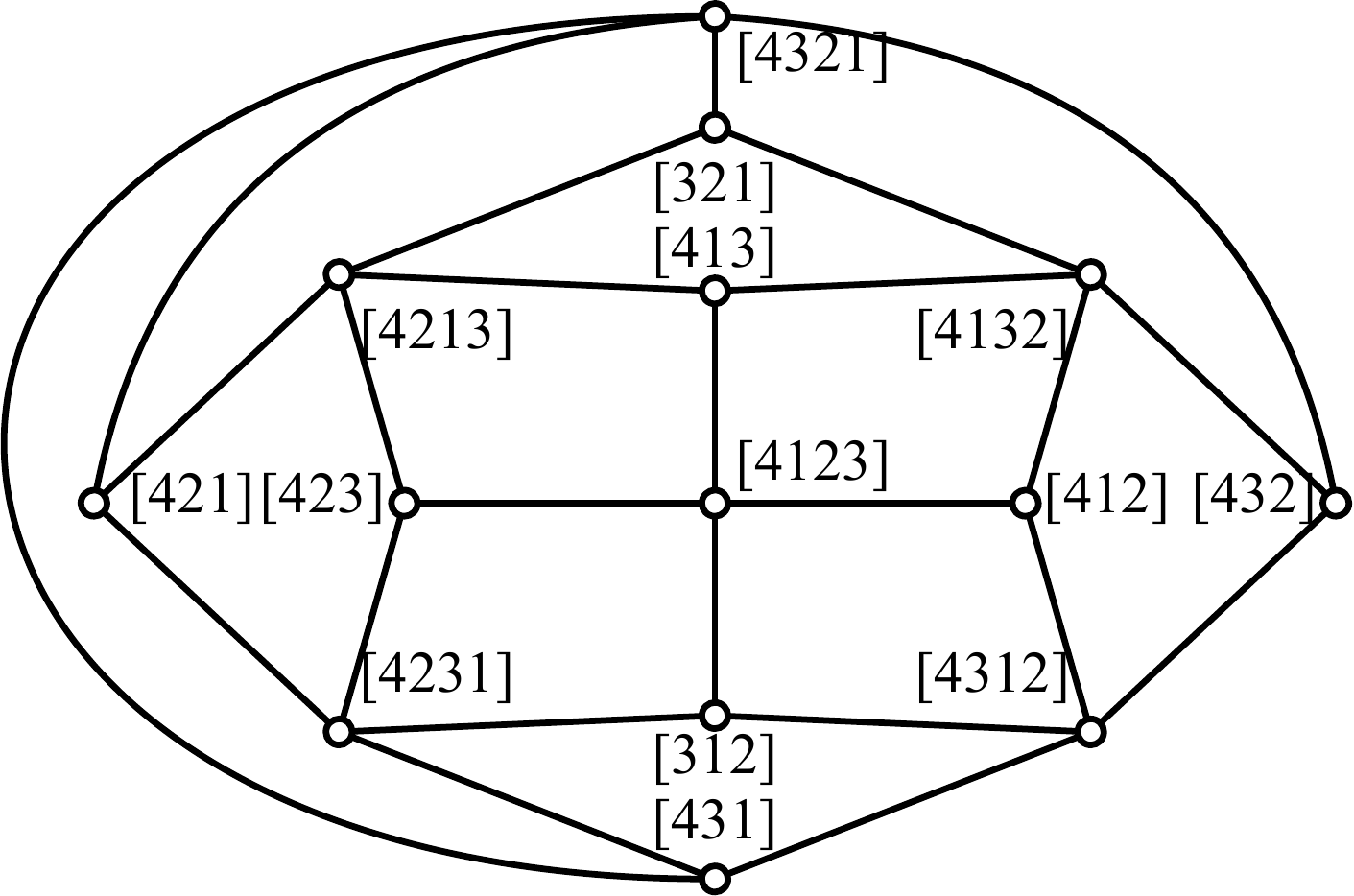}} &
\raisebox{-\height}{\includegraphics[scale=0.25]{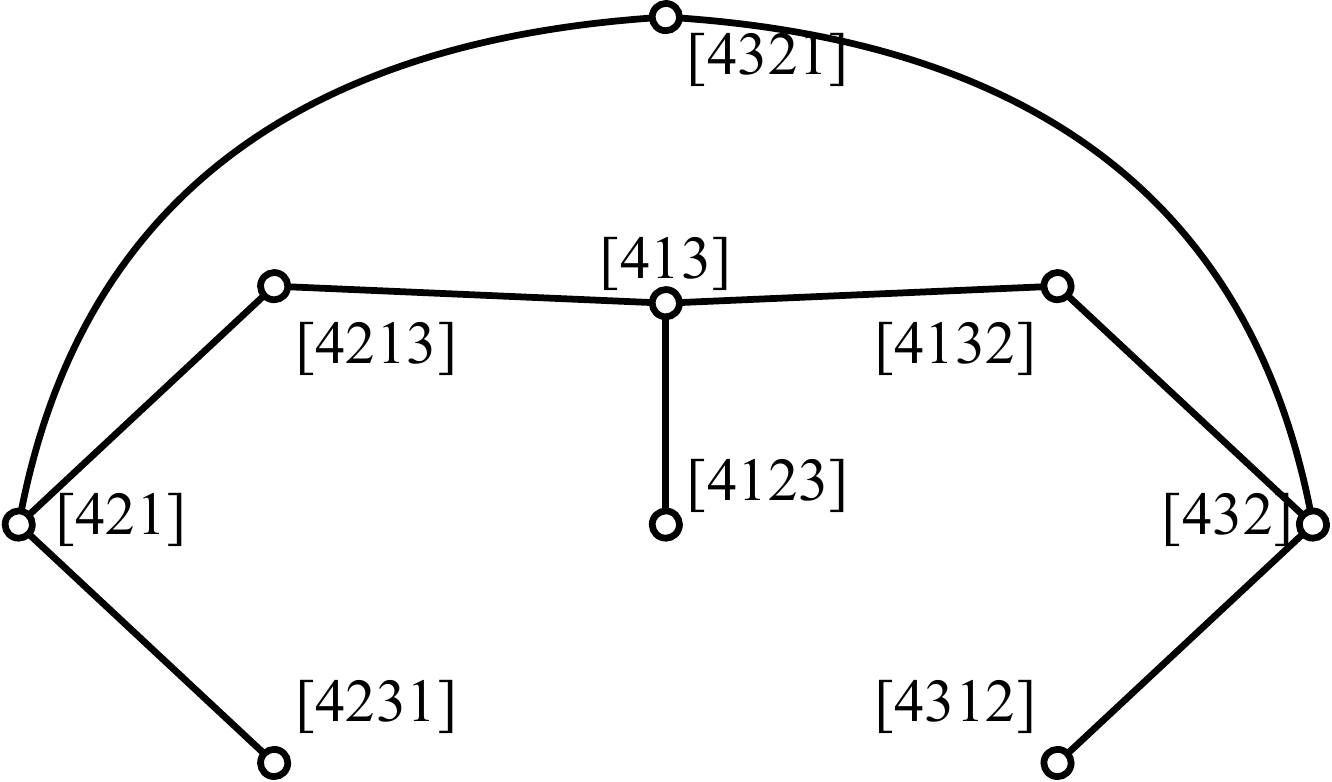}}
\vspace{-0.1in} \\ 
 \hspace{1.1in} a) $\Wilf{4}$ & b) $\WilfCover{4}$
\end{tabular}
\vspace{-0.1in}
\caption{a) Wilf's rotation system $\Wilf{4}$, and b) the induced bifacial rotation system $\Wilf{4}[\quotient{\PERMS{n}} \cup \{[413], [421], [432]\}]$. 
}
\label{fig:Wilf4}
\end{figure}

To appreciate $\Wilf{n}$, we must consider how the basic cycles of $\TheDigraph{n}$ interact.
If $S \in \SigmaCycles{n}$ and $A \in \AltCycles{n}$ intersect, then $S \cap A = \{(\mathbf{p},\mathbf{p} \sigma)\}$, and we let their \emph{sink} be $\sink{S}{A}=\mathbf{p} \sigma$.
For example, $\sink{s([4123])}{a([413])} = 2341$ by Figure \ref{fig:Intersect4}.
Remark \ref{rem:intersectSink} equates cycle intersection with equivalence class consistency, and provides a formula for the sink. 
For simplicity, we let $\sink{[\mathbf{p}]}{[\mathbf{q}]} = \sink{s([\mathbf{p}])}{a([\mathbf{q}])}$.

\begin{remark} \label{rem:intersectSink}
$s(X) \in \SigmaCycles{n}$ and $a(Y) \in \AltCycles{n}$ intersect if and only if $X$ and $Y$ are consistent.
Moreover, if $X = [p_1 p_2 \tdots p_n]$ and $Y = [p_2 p_3 \tdots p_n]$, then $\sink{X}{Y} =~p_1 p_2 \tdots p_n$.
\end{remark}

\begin{figure}[h]
\begin{tabular}{@{}c@{\hspace{0.35in}}c@{}}
\raisebox{0.4\height}{\includegraphics[width=0.35\textwidth]{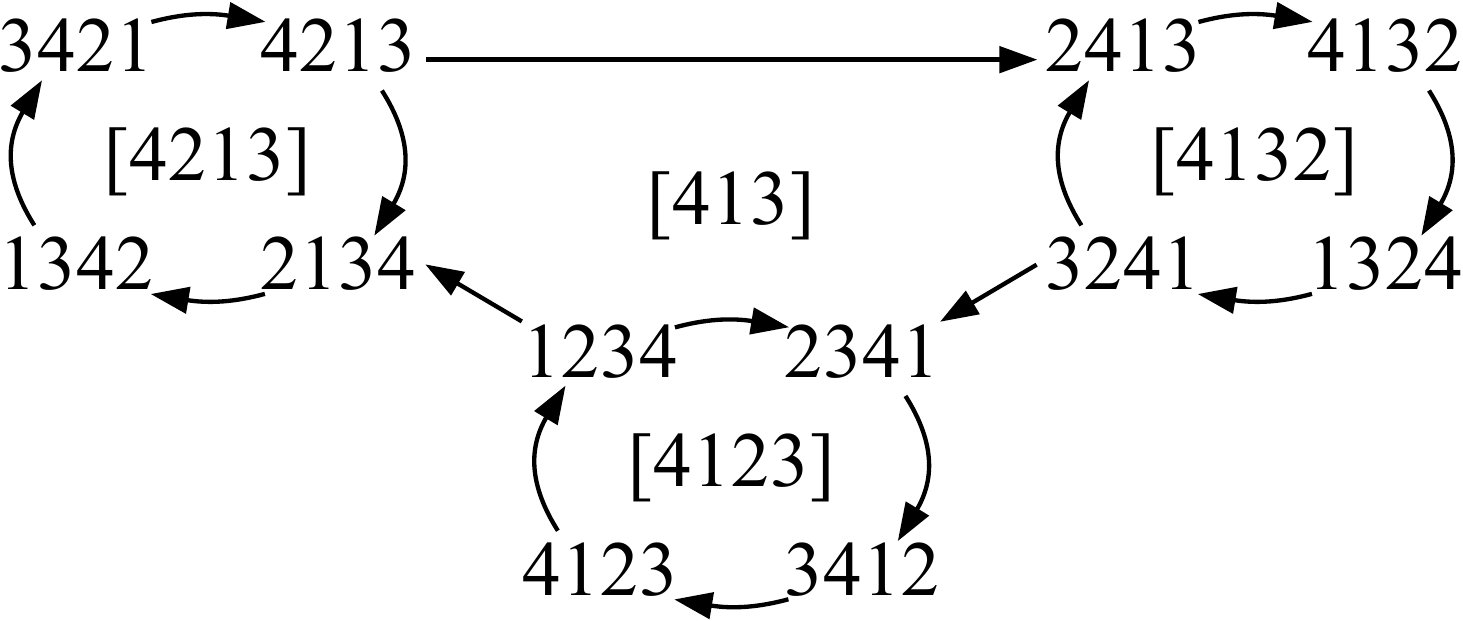}} &
\includegraphics[width=0.35\textwidth]{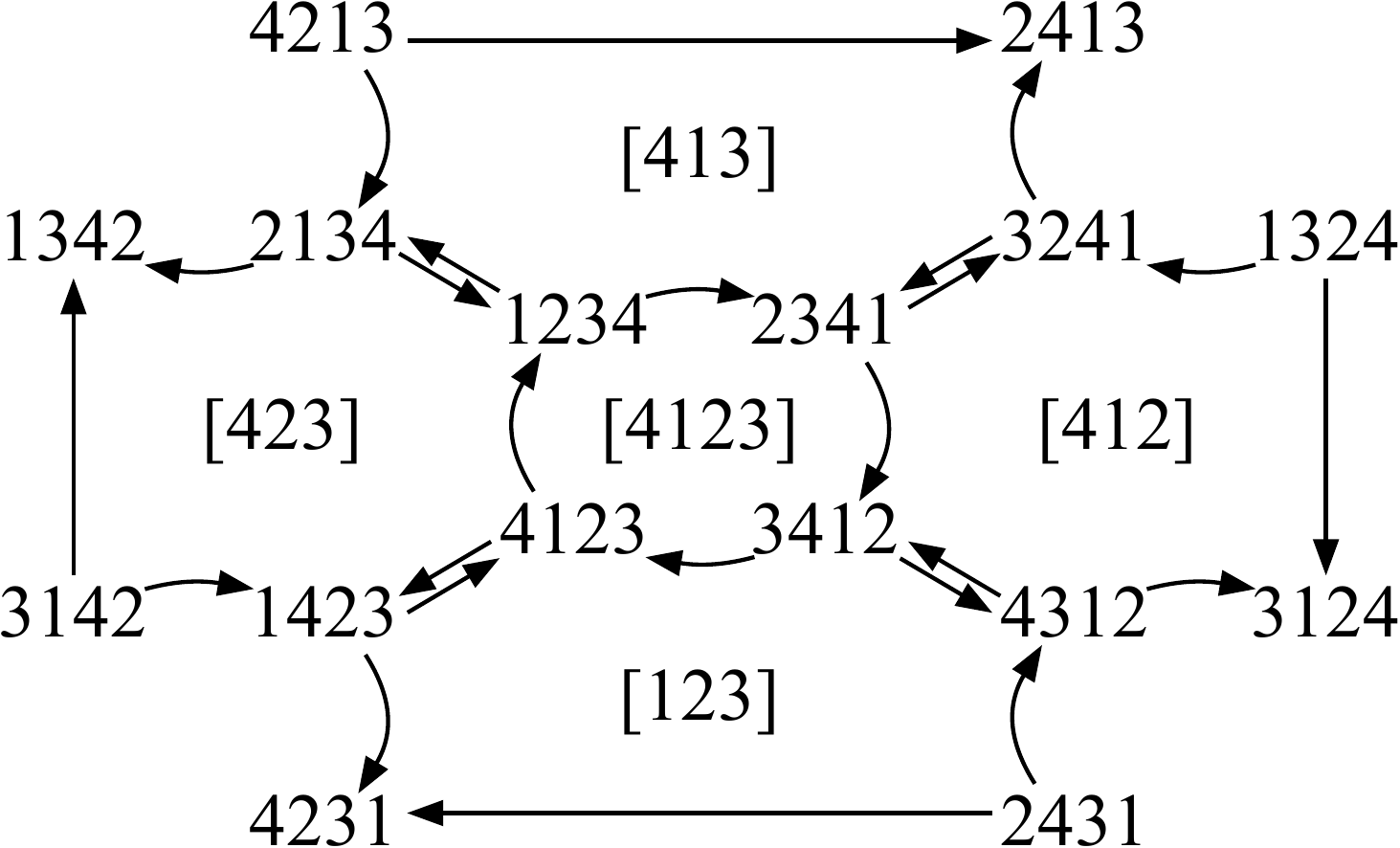} \\ 
\vspace{-0.25in} \\ 
a) & b)
\end{tabular}
\vspace{-0.1in}
\caption{
a) $a([413])$ intersects $s([4213])$, $s([4123])$, $s([4132])$, and b) $s([4123])$ intersects $a([123])$, $a([423])$, $a([413])$, $a([412])$.
}
\label{fig:Intersect4}
\end{figure}

The \emph{induced rotation system} for $R=(V,E,\theta)$ and $U \subseteq V$ is $R[U] = (V',E',\theta')$ with $V' = U'$, $E' = \{e \in E : e=(x,y) \text{ for } x,y \in U\}$, and $\theta'(x) = \theta(x)$ $\forall x \in V'$ with removed edges omitted. 
For example, see Figure \ref{fig:Wilf4} b). 
Lemma \ref{lem:equal}'s map is illustrated by Figure \ref{fig:Map4}, where $F$ is a face in Figure \ref{fig:Wilf4} b) and $C$ is a cycle in Figure~\ref{fig:TheGraphTheCover4}~b). 

\begin{figure}
\begin{tabular}{@{}c@{\;}c@{}}
\raisebox{-0.5\height}{\includegraphics[scale=0.26]{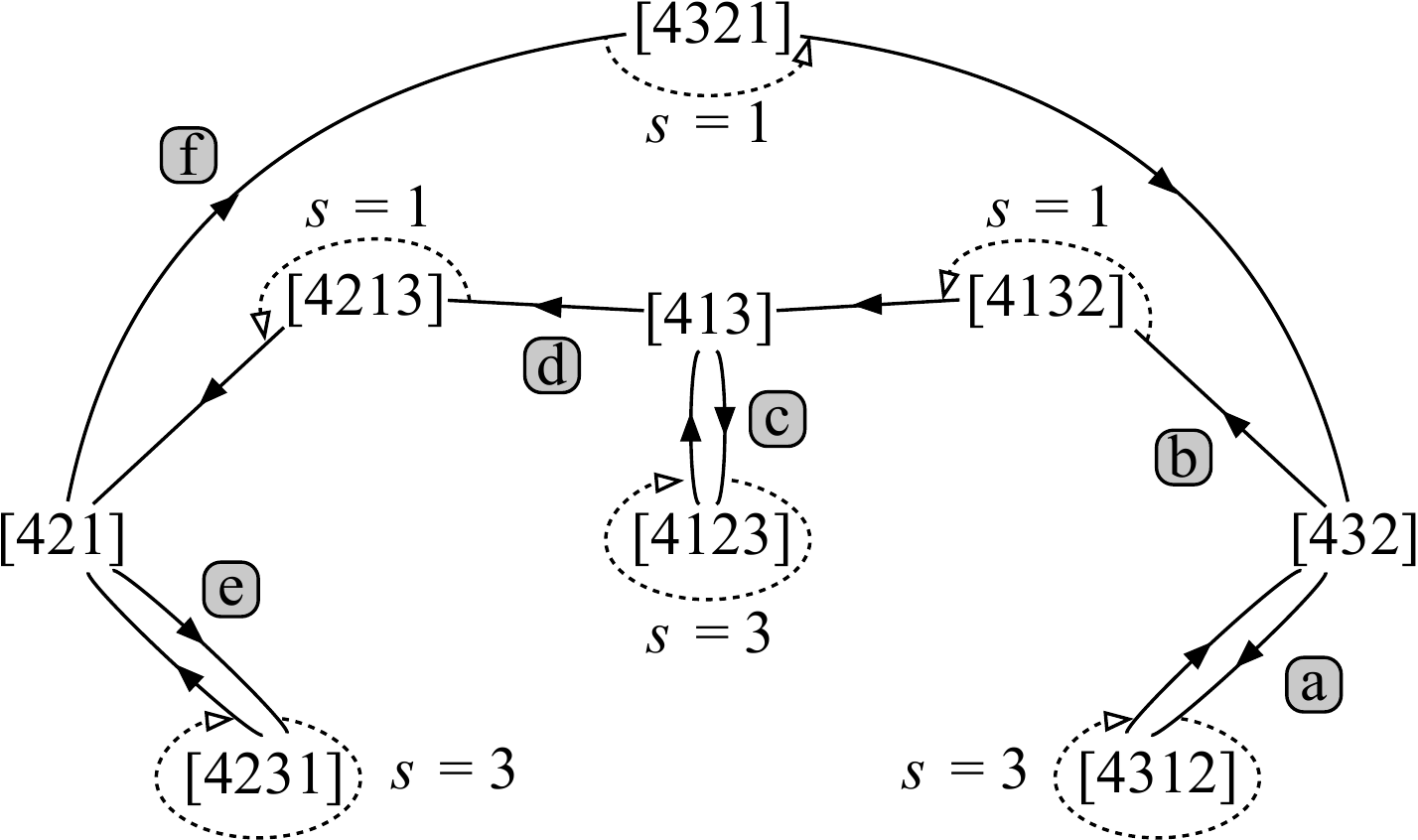}} & 
\raisebox{-0.5\height}{\includegraphics[scale=0.26]{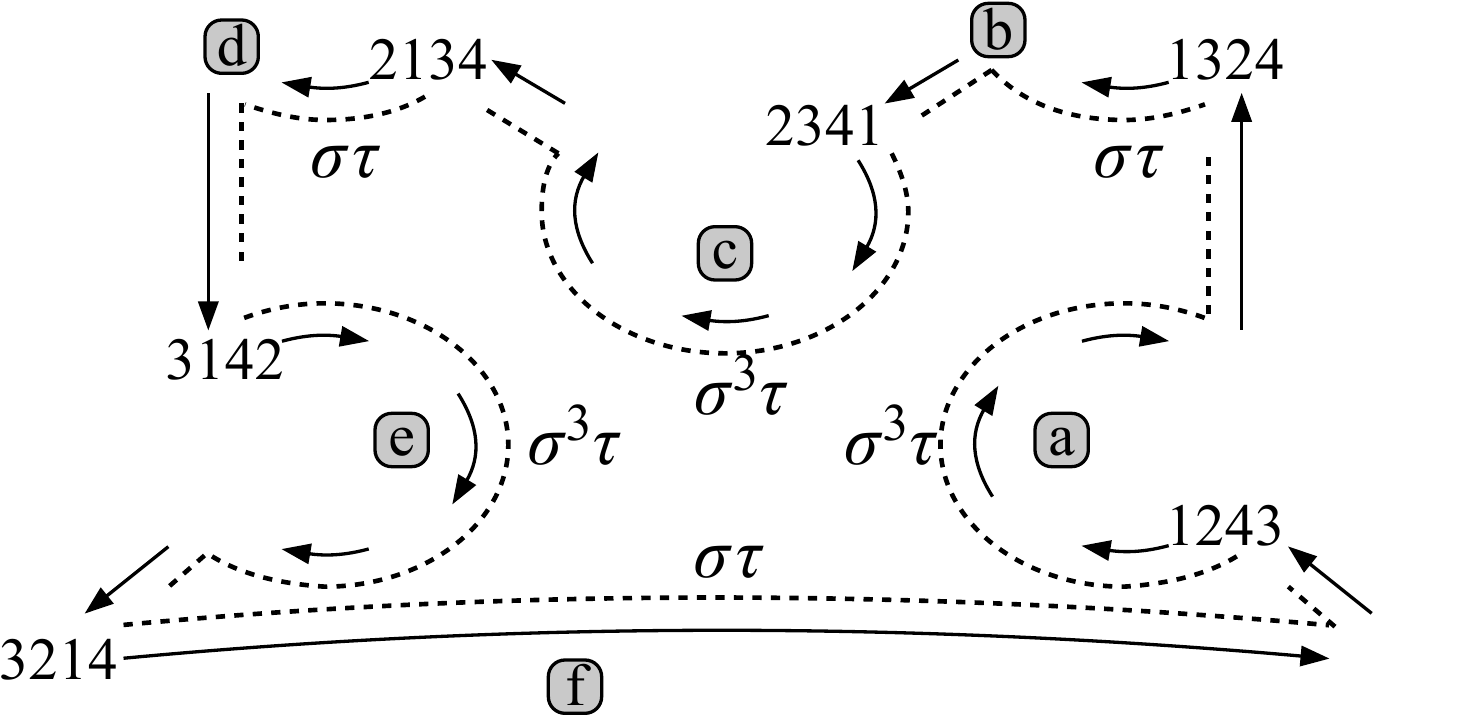}} 
\vspace{-0.2in} \\ 
a) & b) 
\end{tabular}
\vspace{-0.1in}
\caption{
a) A face in $\WilfCover{4}$ maps to b) a cycle in $\SigmaEdges \xor \ACover{4}$.
For example, $s=3$ edges from $\Wilf{4}$ are `skipped' after $\dart{[432]}{[4312]}$, so this dart in a) maps to $\walk{1243 \sigma^3 \tau}$ in b) by \protect\includegraphics[width=0.8em]{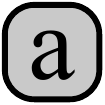}. % where $1243 = \walk{\sink{[432]}{4312}}$.
%In general, darts $\dart{Y}{X}$ map to $\walk{\sink{X}{Y} \sigma^s \tau}$ for $X \in \quotient{\PERMS{4}}$ and $Y \in \quotient{\MISSING{4}}$.
}
\label{fig:Map4}
\end{figure}

\begin{lemma} \label{lem:equal}
The size of $D=\SigmaCycles{n} \xor (a(Y_1) \cup a(Y_2) \cup \tdots \cup a(Y_h))$ is the number of faces plus  isolated vertices in $\Wilf{n}[\quotient{\PERMS{n}} \cup \mathbf{Y}]$ for any $\mathbf{Y} = \{Y_1, Y_2, \sdots, Y_h\} \subseteq \quotient{\MISSING{n}}$.
\end{lemma}
\begin{proof}
$D$ is a cycle cover by Lemma \ref{lem:xor}.
Let $W = \Wilf{n}$ with edge order $\theta$ and $W' = \Wilf{n}[\quotient{\PERMS{n}} \cup \mathbf{Y}]$ with edge order $\theta'$.
Vertex $[\mathbf{p}] \in \quotient{\PERMS{n}}$ is isolated in $W'$ if and only if the $\sigma$-cycle $s([\mathbf{p}])$ is in $D$ by Remark \ref{rem:intersectSink}, and no vertices in $\mathbf{Y}$ are isolated.
Thus, bijecting the faces in $W'$ with the cycles in $D$ that have a $\tau$-edge will prove the lemma.
Consider a face with $[\mathbf{q}^i] \in \mathbf{Y}$ and $[\mathbf{p}^i] \in \quotient{\PERMS{n}}$ for $i \in \nset{k}$
\begin{equation*}
F = 
[\mathbf{q}^1] \ \ e_1 \ \ 
[\mathbf{p}^1] \ \ f_1 \ \ 
[\mathbf{q}^2] \ \ e_2 \ \
[\mathbf{p}^2] \ \ f_2 \ \ 
\tdots \ \ f_{k-1} \ \ 
[\mathbf{q}^k] \ \ e_k \ \
[\mathbf{p}^k] \ \ f_k.
\end{equation*}
Notice that each $e_i$ is immediately followed by $f_i$ in $\theta'(\mathbf{p}^i)$.
However, this may not be true in $\theta(\mathbf{p}^i)$, and we let $s_i$ count the intermediate edges from $e_i$ to $f_i$ in $\theta(\mathbf{p}^i)$.
In other words, $s_i$ edges from $W$ are ``skipped over" on the $e_i \ [\mathbf{p}^i] \ f_i$ portion of~$F$, where $0 \leq s_i \leq n{-}1$.
We claim that $D$ contains the following directed cycle
\begin{equation*}
C = \walk{\sink{\mathbf{p}^1}{\mathbf{q}^1} \, \sigma^{s_1} \, \tau \, \sigma^{s_2} \, \tau \ \tdots \ \sigma^{s_k} \, \tau}.
\end{equation*}
To prove this claim, we show that $\walk{\sink{\mathbf{p}^i}{\mathbf{q}^i} \, \sigma^{s_i} \, \tau}$ ends at $\sink{\mathbf{p}^{i+1}}{\mathbf{q}^{i+1}}$ for $i \in \nset{k}$ and each edge is in $D$. 
By Remark \ref{rem:intersectSink}, there exists $p_1 p_2 \tdots p_n \in \PERMS{n}$ such~that
\begin{equation*} 
[\mathbf{p}^i] = [p_1 p_2 \tdots p_n] \text{ and } [\mathbf{q}^i] = [p_2 p_3 \tdots p_n] \text{ and } \sink{[\mathbf{p}^i]}{[\mathbf{q}^i]} = p_1 p_2 \tdots p_n.
\end{equation*}
To simplify notation, let $s = s_i$.
By $s$ and $e_i$ and the definition of $\theta([\mathbf{p}^{i}])$ from \eqref{eq:thetap},
\begin{equation}  
[p_1 p_3 p_4 \tdots p_n], [p_1 p_2 p_4 p_5 \tdots p_n], \sdots, [p_1 \tdots p_s p_{s+2} p_{s+3} \tdots p_n] \notin \mathbf{Y} \text{ if } s>0, \label{eq:notin}
\end{equation}
and the next vertex in this sequence is $[\mathbf{q}^{i+1}] \in \mathbf{Y}$ below
\begin{align*}
[\mathbf{q}^{i+1}] &= 
\begin{cases}
[p_1 \tdots p_{s+1} p_{s+3} p_{s+4} \tdots p_n] = [p_{s+1} p_{s+3} p_{s+4} \tdots p_n p_1 p_2 \tdots p_{s}] & \text{ if } s<n{-}1 \\
[\mathbf{q}^{i}] = [p_2 p_3 \tdots p_n]= [p_n p_2 p_3 \tdots p_{n-1}] & \text{ if } s=n{-}1.
\end{cases}
\end{align*}
Since $W'$ includes all $\quotient{\PERMS{n}}$ vertices, $f_i$ is followed by the same edge in $\theta'([\mathbf{p}^{i}])$ and $\theta([\mathbf{p}^{i}])$.
Therefore, by the definition of $\theta([\mathbf{q}^{i+1}])$ from \eqref{eq:thetaq}, $[\mathbf{p}^{i+1}]$ is equal to
\begin{align*}
\begin{cases}
[p_1 p_2 \tdots p_s p_{s+2} p_{s+1} p_{s+3} p_{s+4} \tdots p_n] = [p_{s+2} p_{s+1} p_{s+3} p_{s+4} \tdots p_n p_1 p_2 \tdots p_{s}] & \text{ if } s < n{-}1 \\
[p_2 p_3 \tdots p_{n-1} p_1 p_n] = [p_1 p_n p_2 p_3 \tdots p_{n-1}] & \text{ if } s = n{-}1.
\end{cases}
\end{align*}
Now we prove that $\sink{[\mathbf{p}^i]}{[\mathbf{q}^i]} \, \sigma^{s} \, \tau$ ends at the desired vertex
\begin{equation*}
p_1 p_2 \tdots p_n \, \sigma^{s} \, \tau = 
\begin{cases}
p_{s+2} p_{s+1} p_{s+3} p_{s+4} \tdots p_n p_1 p_2 \tdots p_{s} = \sink{[\mathbf{p}^{i+1}]}{[\mathbf{q}^{i+1}]} & \text{ if } s<n{-}1 \\
p_1 p_n p_2 p_3 \tdots p_{n-1} = \sink{[\mathbf{p}^{i+1}]}{[\mathbf{q}^{i+1}]} & \text{ if } s = n{-}1.
\end{cases}
\end{equation*}
Finally, the edges on this path are in $W'$ by \eqref{eq:notin} and $[\mathbf{q}^{i}],[\mathbf{q}^{i+1}] \in \mathbf{Y}$.
Therefore, the claim is true.
This mapping from $F$ to $C$ is invertible and provides our bijection.
\end{proof}

\section{Hamilton Paths and Cycles} \label{sec:Hamilton}

This section proves that $\TheCycle{n}$ is a Hamilton cycle, and $\TheCover{n}$ is a disjoint cycle cover of size two. 
By Lemmas \ref{lem:rules} and \ref{lem:equal}, we must prove that the following induced rotation systems have no isolated vertices and are unifacial and bifacial, respectively
\begin{equation} \label{eq:Wilfs}
\WilfCycle{n} = \Wilf{n}[\quotient{\PERMS{n}} \cup \YCycle{n}] \text{ and } \WilfCover{n} = \Wilf{n}[\quotient{\PERMS{n}} \cup \YCover{n}].
\end{equation}
We begin with two preliminary steps.
First, we show that $\WilfCycle{n}$ and $\WilfCover{n}$ are the sparsest induced rotation systems that are unifacial and bifacial, respectively.
Second, we reduce $\WilfCycle{n}$ and $\WilfCover{n}$ to simpler rotation systems without changing their number of faces.
Besides our main results, we create the Hamilton path $\ThePath{n}$ from $\TheCover{n}$, and show that $\TheCycle{n}$ has size two when $n$ is even.

\subsection{Edge Surplus} \label{sec:Hamilton_surplus}

Determining the number of faces in a rotation system $R$ is simplified when its graph $B=(V,E)$ is close to a tree.
We say that $R$ is \emph{connected} if $G$ is connected, and its \emph{edge surplus} is $|E|-|V|+1$. 
(The connected results of Remark \ref{rem:facesTree} follow from Lemma \ref{lem:reduce} in Section \ref{sec:Hamilton_reduce} by reductions to an edge and a~loop.)

\begin{remark} \label{rem:facesTree}
If $R$ is connected with edge surplus $0$ or $1$, then $R$ has $1$ or $2$ faces, respectively.
If $R$ is disconnected with edge surplus $-s$, then $R$ has at least $s{+}1$~faces.
\end{remark}

Now consider the edge surplus of induced rotation systems of $\Wilf{n}$.

\begin{lemma} \label{lem:surplusInduced}
$\Wilf{n}[\quotient{\PERMS{n}} \cup \mathbf{Y}]$ has edge surplus $(n{-}2)|\mathbf{Y}|{-}(n{-}1)!{+}1$ for ${\mathbf{Y} \subseteq \quotient{\MISSING{n}}}$.
\end{lemma}
\begin{proof}
The graph has $(n{-}1)!{+}|\mathbf{Y}|$ vertices, and each vertex in $\mathbf{Y}$ has degree $n{-}1$ by Remark~\ref{rem:intersectSink}.
Thus, the edge surplus is $(n{-}1)|\mathbf{Y}|{-}((n{-}1)!{+}|\mathbf{Y}|){+}1$. 
\end{proof}

\begin{corollary} \label{cor:surplusThe}
The edge surplus of $\WilfCycle{n}$ and $\WilfCover{n}$ is $n{-}1$ and $1$, respectively.
\end{corollary}
\begin{proof}
Notice $|\FixSubset[n]{r}{m}| = (n{-}3)!$ for each $r$ and $m$.
Thus, $|\YCycle{n}| = (n{-}1)(n{-}3)!{+}1$ and $|\YCover{n}| = (n{-}1)(n{-}3)!$, and so the results follow from Lemma \ref{lem:surplusInduced}. 
\end{proof}

These results imply that $\YCycle{n}$ and $\YCover{n}$ are as small as possible.
More specifically, if $|\mathbf{Y}| < (n{-}1)(n{-}3)! = |\YCover{n}|$, then $\Wilf{n}[\quotient{\PERMS{n}} \cup \mathbf{Y}]$ has at least $n{-}2$ faces by Lemma \ref{lem:surplusInduced} and Remark \ref{rem:facesTree}.
Similarly, if $|\mathbf{Y}| = (n{-}1)(n{-}3)! < |\YCycle{n}|$, then $\Wilf{n}[\quotient{\PERMS{n}} \cup \mathbf{Y}]$ has at least two faces.
Thus, $\WilfCycle{n}$ and $\WilfCover{n}$ are the sparsest unifacial and bifacial induced rotation systems, respectively. 

Remark \ref{rem:facesTree} implies that we only need to prove that $\WilfCover{n}$ is connected in order to prove that it is bifacial.
Remark \ref{rem:surplusWheeel} is used when proving that $\WilfCycle{n}$ is unifacial.

\begin{remark} \label{rem:surplusWheeel}
The edge surplus of the spinning-wheeel $\Wheeel{m}$ is $m$.
\end{remark}

\subsection{Reductions} \label{sec:Hamilton_reduce}

Consider the following operations on $R = (V,E,\theta)$ with $v \in V$
\setlength{\leftmargini}{0.4\leftmargini}
\begin{itemize}
\item If $v$ has degree one and is adjacent to $u$, then \emph{deleting $v$} is $R \backslash v = (V',E',\theta')$ for $V  \,{=}\, V \backslash \{v\}$, $E'  \,{=}\, E \backslash \{(v,u)\}$, and $\theta'(x)  \,{=}\, \theta(x)$ $\forall x \in V'$ except $\theta'(u)$ omits $(u,v)$.
\item If $v$ has degree two and is adjacent to distinct $x$ and $y$, then \emph{smoothing $v$} is $R / v = (V',E',\theta')$ for $V  \,{=}\, V \backslash \{v\}$, $E'  \,{=}\, E \backslash \{(v,x),(v,y)\}$, and $\theta'(x)  \,{=}\, \theta(x)$ $\forall x \in V'$ except $\theta'(x)$ omits $(v,x)$ and $\theta'(y)$ omits $(v,y)$.
\end{itemize}
\setlength{\leftmargini}{1.0\leftmargini}
If $R'$ is obtained from $R$ by a series of these operations, then $R$ \emph{reduces} to $R'$ and we write $R \succ R'$.
Figure \ref{fig:Path5} shows $\WilfCover{4} \succ \WilfP{4}$ ($\WilfP{4}$ is defined after Lemma~\ref{lem:degree12}). 

\begin{remark} \label{rem:surplus}
If $R \succ R'$, then they have the same edge surplus.
\end{remark}

\begin{remark} \label{rem:components}
If $R \succ R'$, then they have an equal number of connected components.
\end{remark}

\begin{lemma} \label{lem:reduce}
If $R \succ R'$, then they have the same number of faces.
\end{lemma}
\begin{proof}
Let $R=(V,E,\theta)$ and consider one operation.
If $v$ has degree one, then a face in $R$ contains $u \ (u,v) \ v \ (v,u) \ u$ while an otherwise identical face in $R \backslash v$ contains $u$.
If $v$ has degree two, then a face in $R$ has $x \ (x,v) \ v \ (v,y) \ y$ while an otherwise identical face in $R / v$ has $x \ (x,y) \ y$, and a face in $R$ has $y \ (y,v) \ v \ (v,x) \ x$ while an otherwise identical face in $R / v$ has $y \ (y,x) \ x$.
All other faces are unchanged.
\end{proof}

\begin{figure}[h]
\begin{tabular}{@{}cc@{}}
\raisebox{-0.5\height}{\includegraphics[scale=0.36]{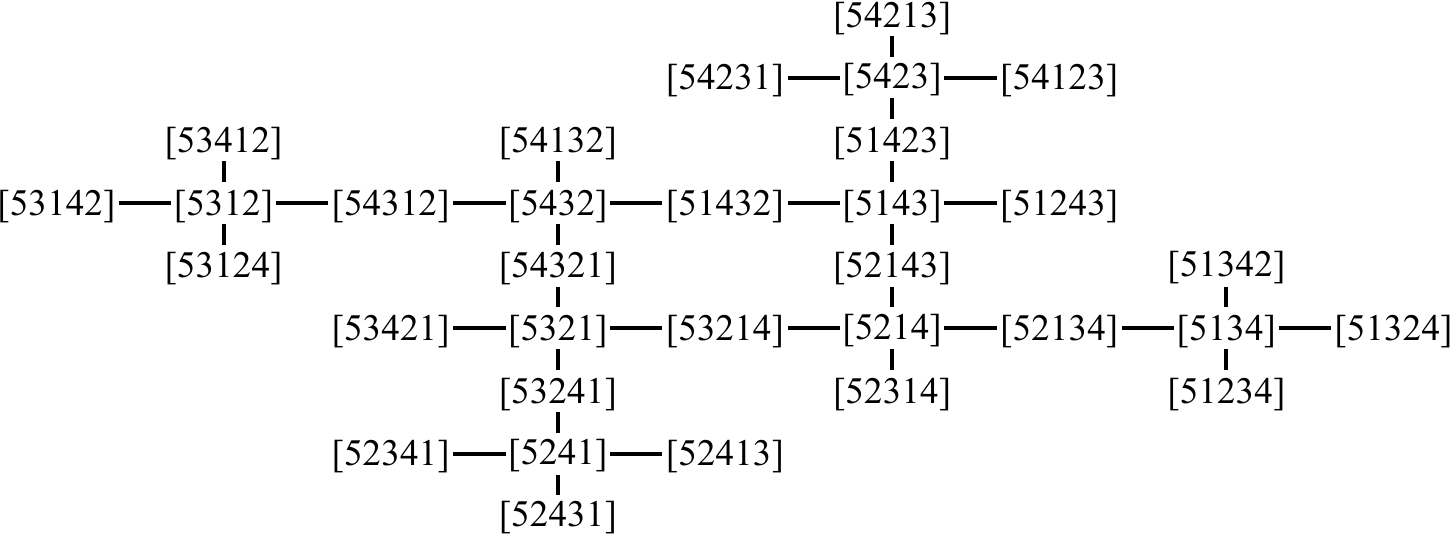}} & 
\raisebox{-0.5\height}{\includegraphics[scale=0.15]{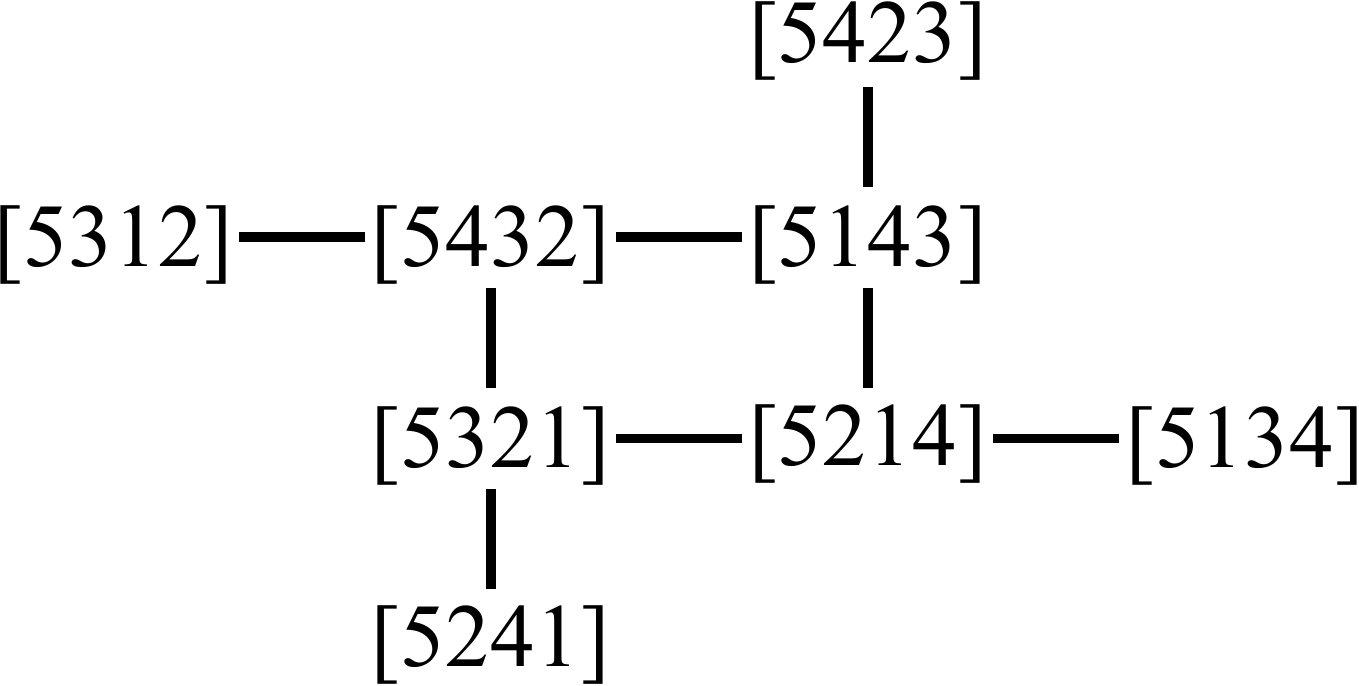}} \\
\vspace{-0.32in} \\ 
\hspace{1.2in} a) $\WilfCover{5}$ & b) $\WilfP{5}$
\end{tabular}
\vspace{-0.1in}
\caption{
a) $\succ$ b) by reducing all $\quotient{\PERMS{5}}$ vertices. 
Edge orders are clockwise.
}
\label{fig:Path5}
\end{figure}

\begin{lemma} \label{lem:degree12}
Each $\quotient{\PERMS{n}}$ vertex has degree $1$ or $2$ in $\WilfCycle{n}$ and $\WilfCover{n}$.
\end{lemma}
\begin{proof}
Consider $\YCover{n}$ and an arbitrary $X = [n p_2 p_3 \tdots p_n] \in \quotient{\PERMS{n}}$.
By Remark \ref{rem:intersectSink}, we must prove that $X$ is consistent with $1$ or $2$ classes in $\YCover{n}$. 
There is a unique~$i$ where $\FixSubset{p_2}{p_i} \subseteq \YCover{n}$.
One consistency with $X$ and $\WilfCover{n}$ involves deleting~$p_i$:
\begin{enumerate}
\item[1)] $X$ and $Y=[n p_2 p_3 \tdots p_{i-1} p_{i+1} p_{i+2} \tdots p_n] \in \FixSubset{p_2}{p_i} \subseteq \YCover{n}$ are consistent.
\end{enumerate}
Notice that $p_2$ is to the right of $n$ in $Y$.
By \eqref{eq:YCover}, the only other consistency with $X$ would involve deleting $p_2$ (and thus changing the symbol to the right of~$n$):
\begin{enumerate}
\item[2)] If $\FixSubset{p_3}{p_2} \subseteq \YCover{n}$, then $X$ and $[n p_3 p_4 \tdots p_n] \in \FixSubset{p_3}{p_2} \subseteq \YCover{n}$ are~consistent.
\end{enumerate}
Thus, $X$ has degree $1$ or $2$ in $\WilfCover{n}$.
Analogous consistencies with $X$ and $\WilfCycle{n}$ are obtained by substituting $\YCycle{n}$ for $\YCover{n}$ in 1) and 2). 
In addition to its $\FixSubset{r}{m}$ subsets, $\WilfCycle{n}$ contains $ [1 2 \tdots n{-}1]$, which leads to the following consistencies:
\begin{enumerate}
\item[3)] If $p_2 p_3 \tdots p_n \in [1 2 \tdots n{-}1]$, then $X$ and $[1 2 \tdots n{-}1] \in \YCycle{n}$ are consistent.
\end{enumerate}
The conditions in 2) and 3) are incompatible, so $X$ has degree $1$ or $2$ in $\WilfCycle{n}$
\end{proof}

Let $\WilfC{n}$ and $\WilfP{n}$ be obtained from $\WilfCycle{n}$ and $\WilfCover{n}$, respectively, by reducing every $\quotient{\PERMS{n}}$ vertex. 
$\WilfC{n}$ and $\WilfP{n}$ are well-defined by Lemma \ref{lem:degree12}, and they have the same number of faces as $\WilfC{n}$ and $\WilfP{n}$ by Lemma \ref{lem:reduce}, respectively.
Figure \ref{fig:Tree6} shows the graphs of $\WilfC{6}$ and $\WilfP{6}$.
Since $\WilfCycle{n}$ and $\WilfCover{n}$ are bipartite, each edge in their reduced system corresponds to a $\quotient{\PERMS{n}}$ vertex.
For example, $([64321],[63251])$ in Figure \ref{fig:Tree6} is from the reduced vertex $[643251] \in \quotient{\PERMS{6}}$.

\begin{figure}[h]
\begin{tabular}{@{}cc@{}}
\raisebox{-1.0\height}{\includegraphics[scale=0.29]{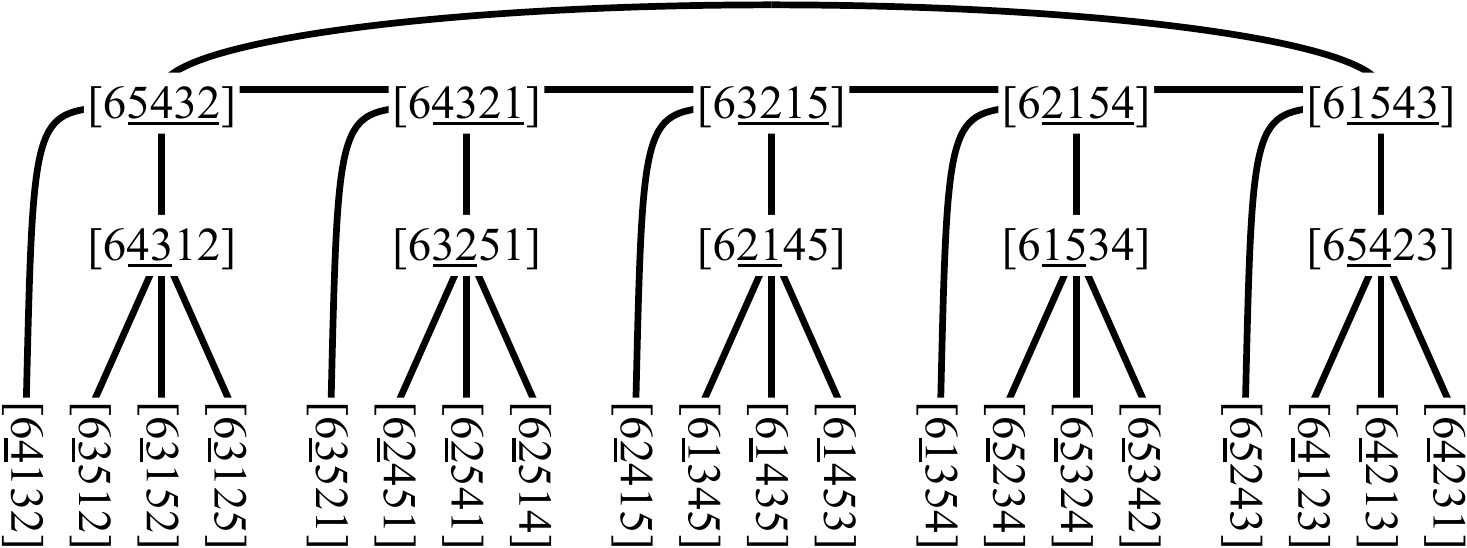}} & 
\raisebox{-1.0\height}{\includegraphics[scale=0.235]{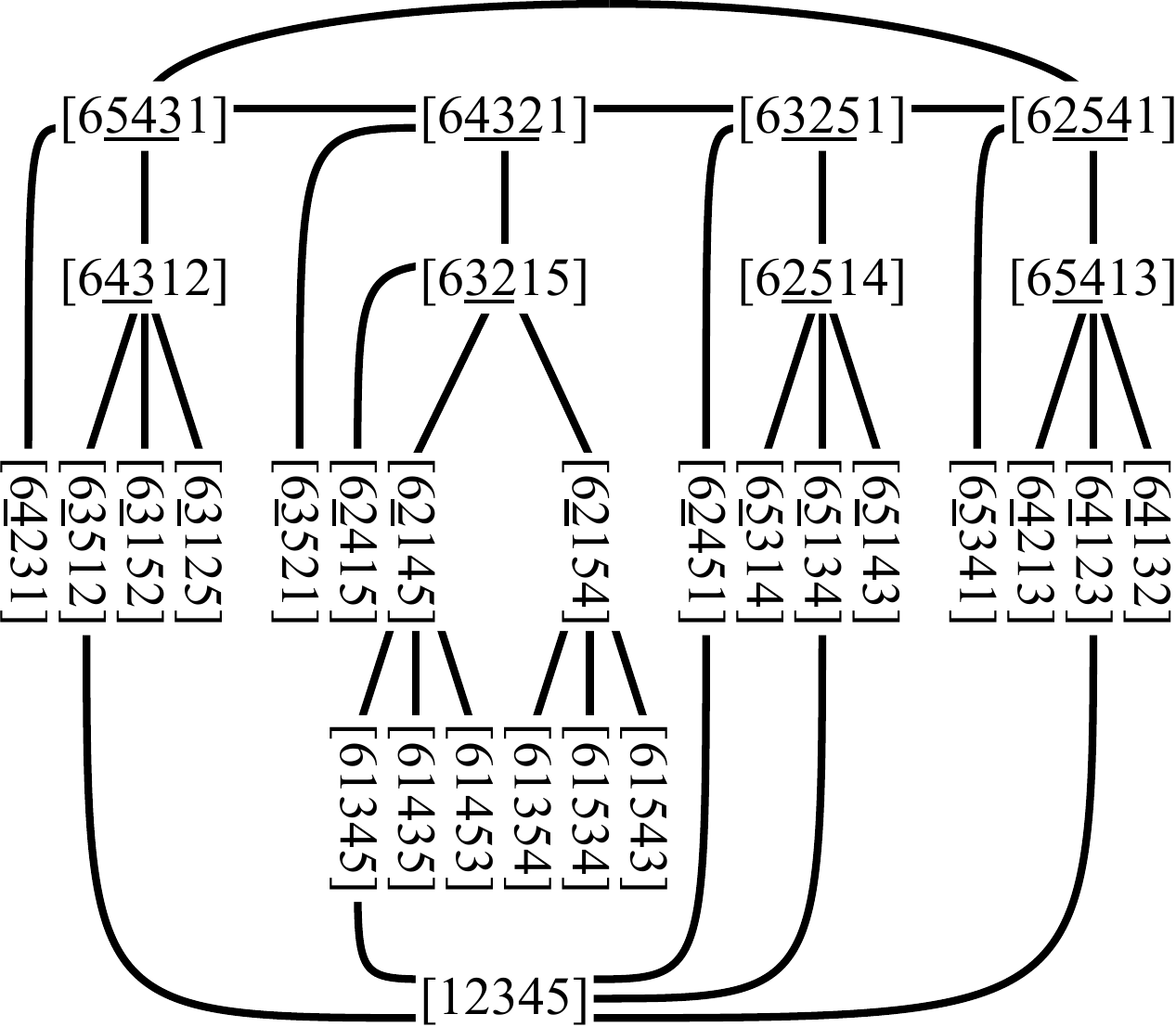}} \\
a) $\WilfP{6}$'s graph & b) $\WilfC{6}$'s graph
\end{tabular}
\vspace{-0.1in}
\caption{The reduced rotation systems with edge orders ignored.
Underlines denote $\underline{p_1 p_2 \tdots p_{\ell}}$ for each $\ell$ in the respective proofs.
}
\label{fig:Tree6}
\end{figure}

The number of faces in a rotation system and an induced rotation system are not necessarily equal.
The following lemma gives a special case where equality holds.

\begin{lemma} \label{lem:inducedEqual}
Suppose $R=(V,E,\theta)$ and $U \subseteq V$.
If $R$ and $R[U]$ are connected and have the same edge surplus, then they have the same number of faces.
\end{lemma}
\begin{proof}
Let $W = V \backslash U$ and $C_i = (W_i, E_i)$ be the connected components of $R[W]$ for $i \in \nset{k}$.
Let $e$ count the edges between $U$ and $W$ in $R$.
By summing over $i \in \nset{k}$,
\begin{align*}
\textstyle |W| = \sum |W_i| \leq \sum (|E_i|-1) = (\sum |E_i|)-k \leq (\sum |E_i|)-e
\end{align*}
where the inequalities follow from the connectedness of $C_i$ and $R$.
Since $R$ and $R[U]$ have the same edge surplus, each inequality holds with equality.
Thus, each $C_i$ is a tree with a single edge connecting it to $R[U]$.
Therefore, $R[U]$ can be obtained from $R$ by a series of degree one deletions.
Hence, the claim follows by Lemma~\ref{lem:reduce}.
\end{proof}

\subsection{Hamilton Path} \label{sec:Hamilton_path}

We now focus on $\TheCover{n}$ the Hamilton path $\ThePath{n}$.
Figure \ref{fig:Tree6}~a) illustrates the cycle $C$ and function $\ell$ from the proof of Theorem \ref{thm:TheCover} for $n=6$.

\begin{theorem} \label{thm:TheCover}
$\TheCover{n}$ is a cycle cover of size two in $\TheDigraph{n}$.
\end{theorem}

\begin{proof}
By Lemmas \ref{lem:rules}, \ref{lem:equal}, \ref{lem:reduce}, and \ref{lem:degree12}, we must prove that $\WilfP{n}$ is bifacial.
By Remarks \ref{rem:facesTree} and \ref{rem:components}, and Corollary \ref{cor:surplusThe}, we only need to prove that $\WilfP{n}$ is connected. Consider any $[\mathbf{p}] = [n p_1 p_2 \tdots p_{n-2}]  \in \YCover{n}$.
The \emph{parent} of $[\mathbf{p}]$ replaces $p_x$ by $p_0$ as follows, 
\begin{equation}  \label{eq:parent}
\parent{[\mathbf{p}]} = [n \, p_0 \, p_1 \, p_2 \, \tdots \, p_{x-1} \, p_{x+1} \, p_{x+2} \, \tdots \, p_{n-2}] \in \FixSubset{p_0}{p_x} \subseteq \YCover{n},
\end{equation}
where $p_0$ and $x$ are unique values such that $[\mathbf{p}] \in \FixSubset{p_1}{p_0}$ and $\FixSubset{p_1}{p_0},\FixSubset{p_0}{p_x} \subseteq \YCover{n}$ from \eqref{eq:YCover}.
Note that $[\mathbf{p}]$ and $\parent{[\mathbf{p}]}$ are adjacent in $\WilfP{n}$ by mutual consistency with reduced vertex $[n \, p_0 \, p_1 \, p_2 \, \tdots p_{n-2}] \in \quotient{\PERMS{n}}$ in $\WilfCover{n}$.
The following cycle in $\WilfP{n}$ repeatedly follows parent edges from the previous vertex,
\begin{align*} 
C = [n \, n{-}1 \, n{-}2 \, \tdots \, 2] \ \parentWalk \ 
[n \, 1 \, n{-}1 \, n{-}2 \, \tdots \, 3] \ \parentWalk \
[n \, 2 \, 1 \, n{-}1 \, n{-}2 \, \tdots \, 4] \\ 
\parentWalk \ \tdots  \ \parentWalk \ 
[n \, n{-}3 \, n{-}4 \, \tdots \, 2 \, 1 \, n{-}1] \ \parentWalk \
[n \, n{-}2 \, n{-}3 \, \tdots \, 2 \, 1] \ \parentWalk.  
\end{align*}
We prove $\WilfP{n}$ is connected by showing that $[\mathbf{p}]$ is on $C$ or has an ancestor on~$C$.
Let $\ell([\mathbf{p}])$ be the largest $k \leq n{-}2$ where $\FixSubset{p_{i+1}}{p_{i}} \subseteq \YCover{n}$ for all $i \in \nset{k{-}1}$. 
Note that $[\mathbf{p}]$ is on $C$ if and only if $\ell([\mathbf{p}]) = n{-}2$; otherwise, $\ell(\parent{[\mathbf{p}]}) > \ell([\mathbf{p}])$.
\end{proof}

We split and join the two cycles of $\TheCover{n}$ to create $\ThePath{n}$.
We show that one of the cycles in $\TheCover{n}$ has length $2(n{-}1)$ and involves $\mathbf{q} = n \, n{-}1 \, \tdots \, 1$ (see Figures \ref{fig:ThePath5} and \ref{fig:TheGraphTheCover4} b)).
Definition \ref{def:ThePath} uses the simplification of Definition \ref{def:TheCover} from Section \ref{sec:intro}, except that $\mathbf{q} \tau$ has no edge entering it, and $\mathbf{q} \sigma$ is entered by a $\sigma$-edge instead of a~$\tau$-edge.

\begin{definition} \label{def:ThePath}
Let $\mathbf{p} \,{=}\, p_0 \tdots p_{n-1}$, $p_i \,{=}\, n$, $r \,{=}\, p_{(i \bmod n{-}1){+}1}$, $\mathbf{q} \,{=}\, n \, n{-}1 \, \tdots \, 1$, and~$\mathbf{p} \,{\neq}\, \mathbf{q} \tau$. 
Then $(\mathbf{p} \tau, \mathbf{p}) \,{\in}\, \ThePath{n}$ if $p_0 = (r \bmod n{-}1){+}1$ and $\mathbf{p} \neq \mathbf{q} \sigma$; 
otherwise, $(\mathbf{p} \sigma^{-1}, \mathbf{p}) \,{\in}\, \ThePath{n}$.
\end{definition} 

\begin{corollary} \label{cor:ThePath}
$\ThePath{n}$ is a Hamilton path from $\mathbf{q} \tau$ to $\mathbf{q} \sigma \tau$ for $\mathbf{q} = n \, n{-}1 \, \tdots \, 1$ in $\TheDigraph{n}$.
\end{corollary}
\begin{proof}
Observe that $\TheCover{n}$ contains the following directed cycle, where $r$ and $p_0$ from Definition \ref{def:TheCover} are underlined and overlined, respectively,
\begin{align*}
\begin{array}{r@{ \ }c@{ \ }llllc}
\mathbf{q}(\tau \sigma)^{n-1} & = &
\overline{n} \, \underline{n{-}1} \, n{-}2 \, \tdots \, 1 & \tau &
\overline{n{-}1} \, n \, \underline{n{-}2} \, n{-}3 \, n{-}4 \, \tdots \, 1 & \sigma \\ &&
\overline{n} \, \underline{n{-}2} \, \tdots \, 1 \, n{-}1 & \tau &
\overline{n{-}2} \, n \, \underline{n{-}3} \, n{-}4 \, \tdots \, 1 \, n{-}1 & \sigma & \ldots \\ &&
\overline{n} \, \underline{1} \, n{-}1 \, n{-}2 \, \tdots \, 2 & \tau & 
\overline{1} \, n \, \underline{n{-}1} \, n{-}2 \, n{-}3 \, \tdots \, 2 & \sigma.
\end{array}
\end{align*}
By comparing Definition \ref{def:TheCover} and \ref{def:ThePath} we have the following
\begin{align} \label{eq:ThePath}
\ThePath{n} 
%&= \TheCover{n} \backslash \{(\mathbf{q} \tau \tau, \mathbf{q} \tau), (\mathbf{q} \sigma \tau, \mathbf{q} \sigma)\} \cup \{(\mathbf{q} \sigma \sigma^{-1}, \mathbf{q} \sigma)\} \notag \\
&= \TheCover{n} \backslash \{(\mathbf{q}, \mathbf{q} \tau), (\mathbf{q} \sigma \tau, \mathbf{q} \sigma)\} \cup \{(\mathbf{q}, \mathbf{q} \sigma)\} \text{ for } \mathbf{q} = n \, n{-}1 \, \tdots \, 1.
\end{align}
Removing $(\mathbf{q}, \mathbf{q} \tau)$ from $\TheCover{n}$ gives a directed path from $\mathbf{q} \tau$ to $\mathbf{q}$, and removing $(\mathbf{q} \sigma \tau, \mathbf{q} \sigma)$ gives a directed path from $\mathbf{q} \sigma$ to $\mathbf{q} \sigma \tau$.
These spanning directed paths are then joined by adding $(\mathbf{q}, \mathbf{q} \sigma)$.
The resulting Hamilton path is from $\mathbf{q} \tau$ to $\mathbf{q} \sigma \tau$.
\end{proof}

\begin{figure}[h]
\includegraphics[width=1.0\textwidth]{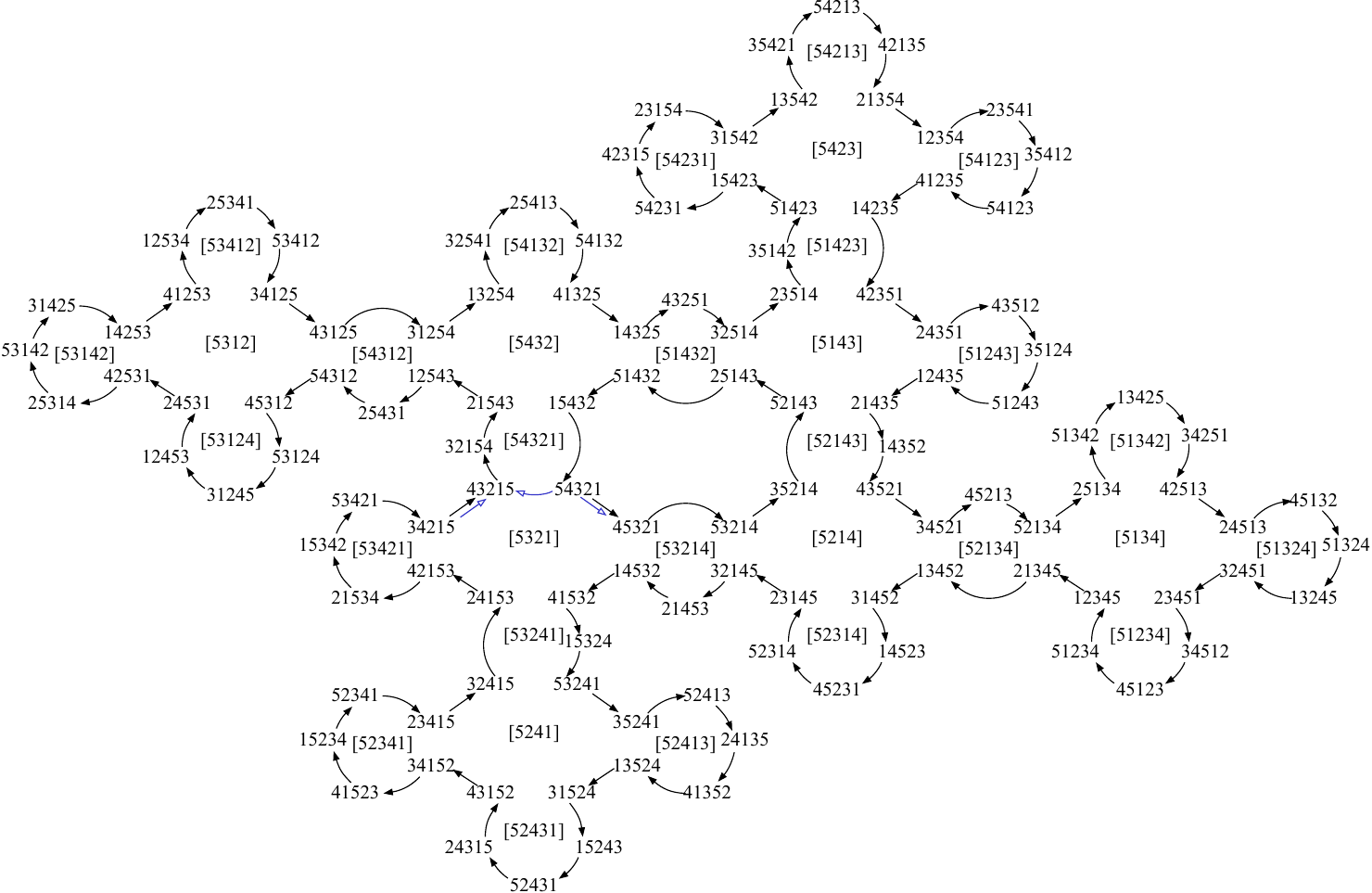}
\vspace{-0.1in}
\caption{
Hamilton path $\ThePath{5}$ from $\mathbf{q} \tau = 45321$ to $\mathbf{q} \sigma \tau = 34215$ for $\mathbf{q} = 54321$ is obtained by complementing the three blue edges in the cycle cover $\TheCover{5}$ above.
This embedding mirrors Figure~\ref{fig:Path5}. % a).
}
\label{fig:ThePath5}
\end{figure}

\subsection{Hamilton Cycle} \label{sec:Hamilton_cycle}

Figure \ref{fig:Tree6} b) shows cycle $R$ and function $\ell$ from the proof of Theorem \ref{thm:TheCycle} for $n=6$, and Figure \ref{fig:WheeelPaths6} shows the induced rotation system for $n=7$.
  
\begin{figure}[h]
\includegraphics[scale=0.4]{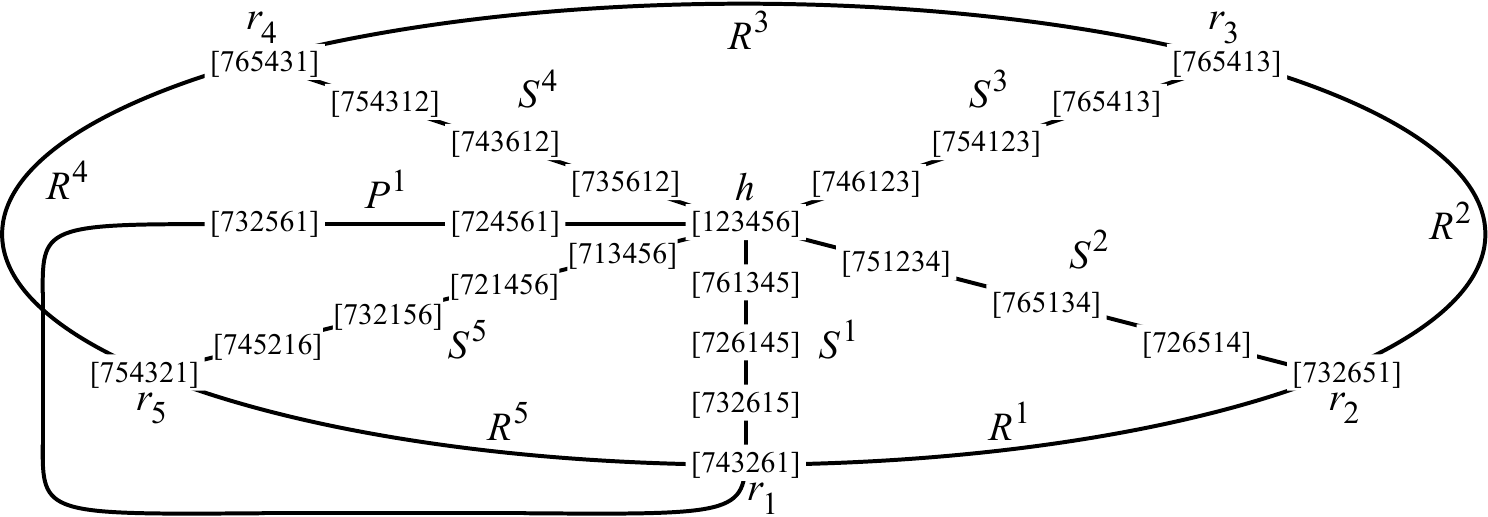}
\vspace{-0.1in}
\caption{
The connected induced rotation system $\WilfC{7}[U]$ from the proof of Theorem \ref{thm:TheCycle}, which reduces to $\Wheeel{6}$.
}
\label{fig:WheeelPaths6}
\end{figure}

\begin{theorem} \label{thm:TheCycle}
$\TheCycle{n}$ is a Hamilton cycle in $\TheDigraph{n}$ when $n$ is odd.
\end{theorem}

\begin{proof}
By Lemmas \ref{lem:rules}, \ref{lem:equal}, \ref{lem:reduce}, and \ref{lem:degree12}, we must prove that $\WilfC{n}$ is unifacial for odd~$n$.
By Remarks \ref{rem:surplusWheeel} and \ref{rem:components}, Lemmas \ref{lem:Wheeel} and \ref{lem:inducedEqual}, and Corollary \ref{cor:surplusThe}, this is true if $\WilfC{n}$ is connected and $\WilfC{n}[U] \succ \Wheeel{n{-}1}$ for some $U \subseteq \YCycle{n}$.
To match the definition of $\Wheeel{n{-}1}$, our $U$ has a hub vertex $h=[1 \, 2 \, \tdots \, n{-}1]$, rim cycle $R$, and spoke paths $S^j$ and $P^1$ from $h$ to $R$. 
We define a parent relationship analogous to \eqref{eq:parent}. 
Consider any $[\mathbf{p}] = [n p_1 p_2 \tdots p_{n-2}]  \in \YCycle{n} \setdiff \{h\}$. 
The \emph{parent} of $[\mathbf{p}]$ replaces $p_x$ by $p_0$ as in \eqref{eq:parent}, where $p_0$ and $x$ are the unique values such that $[\mathbf{p}] \in \FixSubset{p_1}{p_0}$ and $\FixSubset{p_1}{p_0},\FixSubset{p_0}{p_x} \subseteq \YCycle{n}$ from \eqref{eq:YCycle}. 
Note that $[\mathbf{p}]$ and $\parent{[\mathbf{p}]}$ are adjacent in $\WilfC{n}$ by mutual consistency with reduced vertex $[n \, p_0 \, p_1 \, p_2 \, \tdots p_{n-2}] \in \quotient{\PERMS{n}}$ in $\WilfCycle{n}$.
The rim cycle in $\WilfC{n}$ is,
\begin{align*}
R =&\ [n \, n{-}1 \, n{-}2 \, \tdots \, 3 \, 1] \ \parentWalk \ 
[n \, 2 \, n{-}1 \, n{-}2 \, \tdots \, 4 \, 1] \ \parentWalk \ 
[n \, 3 \, 2 \, n{-}1 \, n{-}2 \, \tdots \, 5 \, 1] \\ 
&\ \parentWalk \ \tdots \ \parentWalk \ 
[n \, n{-}3 \, n{-}4 \, \tdots \, 2 \, n{-}1 \, 1] \ \parentWalk \ 
[n \, n{-}2 \, n{-}3 \, \tdots \, 2 \, 1] \ \parentWalk.
\end{align*}
We prove $\WilfC{n}$ is connected by showing that $[\mathbf{p}]$ is on $R$ or has an ancestor on $R$.
Let $\ell([\mathbf{p}])$ be the largest $k \leq n{-}3$ where $\FixSubset{p_{i+1}}{p_{i}} \subseteq \YCycle{n}$ and $p_i \neq 1$ for $i \in \nset{k{-}1}$. 
Note that $[\mathbf{p}]$ is on $R$ if and only if $\ell([\mathbf{p}]) = n{-}3$; otherwise, $\ell(\parent{[\mathbf{p}]}) > \ell([\mathbf{p}])$.
To define our spoke paths from $h$ to $R$ we first label the vertices on $R$ by going the other way around the cycle as follows,
\begin{align*}
r_1 &= [n \, n{-}3 \, n{-}4 \, \tdots \, 2 \, n{-}1 \, 1] & 
    r_{n-4} &= [n \, 2 \, n{-}1 \, n{-}2 \, \tdots \, 4 \, 1] \\
r_2 &= [n \, n{-}4 \, n{-}5 \, \tdots \, 2 \, n{-}1 \, n{-}2 \, 1] &
    r_{n-3} &= [n \, n{-}1 \, n{-}2 \, \tdots \,  3 \, 1] \\
r_3 &= [n \, n{-}5 \, n{-}6 \, \tdots \, 2 \, n{-}1 \, n{-}2 \, n{-}3 \, 1] \ \ldots &
    r_{n-2} &= [n \, n{-}2 \, n{-}3  \, \tdots  \, 2 \, 1].
\end{align*}
We label the incidences with $h$ in $\WilfC{n}$ as follows,
\begin{align*}
e_1 &= (h, [n \, n{-}1 \, 1 \, 3 \, 4 \, \tdots \, n{-}2]) & e_{n-4} &= (h, [n \, 4 \, 6 \, 7 \, \tdots \, n{-}1 \, 1 \, 2 \, 3]) \\
e_2 &= (h, [n \, n{-}2 \, 1 \, 2 \, \tdots \, n{-}3])  & e_{n-3} &= (h, [n \, 3 \, 5 \, 6 \, \tdots \, n{-}1 \, 1 \, 2]) \\
e_3 &= (h, [n \, n{-} 3 \, n{-}1 \, 1 \, 2 \, \tdots \, n{-}4]) &  f &= (h, [n \, 2 \, 4 \, 5 \, \tdots \, n{-}1 \, 1]) \\
& & e_{n-2} &= (h, [n \, 1 \, 3 \, 4 \, \tdots \, n{-}1]).
\end{align*}
The spoke paths are defined below, where $\cdots$ denotes repeated parent edges,
\begin{align*}
& S^1 =  h \; e_1 \; 
[n \; n{-}1 \; 1 \; 3 \; 4 \; \tdots \; n{-}2] \, 
\cdots \, 
[n \; n{-}4 \; n{-}5 \; \tdots \;2 \; n{-}1 \; 1 \; n{-}2] \;  \parentWalk \;
r_{1} \\ 
& S^2 =  h \; e_2 \;
[n \; n{-}2 \; 1 \; 2 \; \tdots \; n{-}3] \,   
\cdots \, 
[n \; n{-}5 \; n{-}6 \; \tdots \; 2 \; n{-}1 \; n{-}2 \; n{-}3 \; 1] \;  \parentWalk \;
r_{2} \\ 
& S^3 =  h \; e_3 \;
[n \; n{-}3 \; n{-}1 \; 1 \; 2 \; \tdots \; n{-}4] \,  
\cdots \, 
[n \; n{-}6 \; n{-}7 \; \tdots \; 2 \; n{-}1 \; n{-}2 \; n{-}3 \; n{-}4 \; 1] \;  \parentWalk \;
r_{3} \\ 
& \ldots =  \ldots \\
& S^{n-4} =  h  \; e_{n-4} \;
[n \; 4 \; 6 \; 7 \; \tdots \; n{-}1 \;1 \; 2 \; 3] \, 
\cdots \,   
[n \; n{-}1 \; n{-}2 \; \tdots \;4 \; 1 \; 3] \;  \parentWalk \;
r_{n-4} \\ 
& S^{n-3}  =  h \; e_{n-3} \;
[n \; 3 \; 5 \; 6 \; \tdots \; n{-}1 \; 1 \; 2] \, 
\cdots \, 
[n \; n{-}2 \; n{-}3 \; \tdots \; 3 \; 1 \; 2] \;  \parentWalk \; 
r_{n-3 } \\ 
& P^{1}  =  h \; f \;
[n \; 2 \; 4 \; 5 \; \tdots \; n{-}1 \; 1] \,  
\cdots \, 
[n \; n{-}4 \; n-5 \; \tdots \; 2 \; n{-}2 \; n{-}1 \; 1] \;  \parentWalk \; 
r_{1} \\ 
& S^{n-2}  =  h \;  e_{n-2} \;
[n \; 1 \; 3 \; 4 \; \tdots \; n{-}1] \,  
\cdots \,  
[n \; n{-}3 \; n{-}4  \; \tdots \; 2 \; 1 \; n{-}1] \;  \parentWalk \;
r_{n-2}. 
\end{align*}
Let $U$ be the union of vertices on these spoke paths. 
By smoothing the paths, $\WilfC{n}[U]$ reduces to a rotation system on the $(n{-}1)$-wheeel graph.
To complete the proof, we must consider edge orders. 
Let the edge orders of $\Wilf{n}$, $\WilfCycle{n}$, $\WilfC{n}$, and $\WilfC{n}[U]$ be $\theta$, $\theta'$, $\theta''$, and $\theta'''$, respectively.
The cyclic orders for $h$ are 
\begin{align*}
\theta(h) = &\
(h, [n \, n{-}1 \, 1 \, 2 \tdots n{-}2]), 
(h, [n \, n{-}2 \, n{-}1 \, 1 \, 2 \tdots n{-}3]), 
(h, [n \, n{-} 3 \, n{-}2 \, n{-}1 \, 1 \, 2 \tdots n{-}4]), \\
&\ \ldots, \ (h, [n \, 3 \, 4 \, \tdots \, n{-}1 \, 1 \, 2]), \,
(h, [n \, 2 \, 3 \, \tdots \, n{-}1 \, 1]), \,
(h, [n \, 1 \, 2 \, \tdots \, n{-}1]) \\
\theta'(h) = &\
e_1, e_2, e_3, \ldots, e_{n-3}, f, e_{n-2} \\
\theta''(h) =&\
\theta'(h) \\
\theta'''(h) =&\
(h, r_1), \ (h, r_2), \ (h, r_3), \ \ldots, \ (h, r_{n-3}), \ (h, r_{1}), \ (h, r_{n-2})
\end{align*}
where the two copies of $(h, r_1)$ in $\theta'''(h)$ are due to $S^1$ and $P^1$, respectively.
We briefly explain: 
$\theta(h)$ inserts $h$'s missing symbol $n$ from right-to-left; 
$\theta'(h)$ is obtained from $\theta(h)$ by replacing each $\quotient{\PERMS{n}}$ vertex with the unique vertex in $\YCycle{n} \setdiff \{h\}$ that is consistent with it;
$\theta''(h) = \theta'(h)$ since each of the incident vertices is in $U$;
$\theta'''(h)$ is obtained from $\theta''(h)$ by reducing the paths.
The cyclic orders for $r_1$ are
\begin{align*}
\theta(r_1) = &\ 
(r_1, [n \, n{-}3 \, n{-}4 \, \tdots \, 2 \, n{-}1 \, 1 \, n{-}2]), \,
(r_1, [n \, n{-}3 \, n{-}4 \, \tdots \, 2 \, n{-}1 \, n{-}2 \, 1]), \\ &\ 
(r_1, [n \, n{-}3 \, n{-}4 \, \tdots \, 2 \, n{-}2 \, n{-}1 \, 1]), \,
\ldots, \,
(r_1, [n \, n{-}2 \, n{-}3 \, n{-}4 \, \tdots \, 2 \, n{-}1 \, 1]) \\
\theta'(r_1) = &\ 
(r_1, [n \, n{-}4 \, n{-}5 \tdots \, 2 \, n{-}1 \, 1 \, n{-}2]), \,
(r_1, [n \, n{-}4 \, n{-}5 \, \tdots \, 2 \, n{-}1 \, n{-}2 \, 1]), \\ &\
(r_1, [n \, n{-}4 \, n{-}5 \, \tdots \, 2 \, n{-}2 \, n{-}1 \, 1]), \,
\ldots, \,
(r_1, [n \, n{-}2 \, n{-}3 \, n{-}4 \, \tdots \, 2 \, 1]) \\
\theta''(r_1)= &\ 
(r_1, [n \, n{-}4 \, n{-}5 \tdots \, 2 \, n{-}1 \, 1 \, n{-}2]), \,
(r_1, r_2), \\ &\
(r_1, [n \, n{-}4 \, n{-}5 \, \tdots \, 2 \, n{-}2 \, n{-}1 \, 1]), \,
(r_1, r_{n-2}) \\
\theta'''(r_1)= &\ 
(r_1, h), \,
(r_1, r_2), \,
(r_1, h), \,
(r_1, r_{n-2})
\end{align*}
where the two copies of $(r_1, h)$ on the last line are due to $S^1$ and $P^1$, respectively.
We briefly explain: 
$\theta(r_1)$ inserts $r_1$'s missing symbol $n{-}2$ from right-to-left; 
$\theta'(r_1)$ is obtained from $\theta(r_1)$ by replacing each $\quotient{\PERMS{n}}$ vertex with the unique vertex in $\YCycle{n} \setdiff \{r_1\}$ that is consistent with it;
$\theta''(r_1)$ is obtained from $\theta'(r_1)$ by including only those edges that have both incident vertices in $U$; 
$\theta'''(h)$ is obtained from $\theta''(h)$ by reducing $S^1$ and $P^1$, respectively.
By similar arguments, 
\begin{align*}
\theta'''(r_i) = &\
(r_1, h), \,
(r_{i}, r_{i+1}), \,
(r_{i}, r_{i-1})
\text{ for } 2 \leq i \leq n{-}2.
\end{align*}
Therefore, the edge orders are correct, and so $\WilfC{n}[U] \succ \Wheeel{n{-}1}$.
\end{proof}

\begin{corollary} \label{cor:fewestTau}
$\TheCycle{n}$ uses the fewest $\tau$-edges of any Hamilton cycle in $\TheDigraph{n}$, and $\TheCover{n}$ uses the fewest $\tau$-edges of any cycle cover of size two in $\TheDigraph{n}$.
\end{corollary}
\begin{proof}
The number of $\tau$-edges in $D=E_\sigma \xor (A_1 \cup A_2 \cup \tdots \cup A_h)$ is $(n{-}1)h$, and $h$ is minimized by the discussion in Section \ref{sec:Hamilton_surplus}.
Thus, the result is due to Lemma~\ref{lem:xor}.
\end{proof}

Our last remark is due to the proof of Theorem \ref{thm:TheCycle} and the parity of Lemma~\ref{lem:Wheeel}.

\begin{remark} \label{rem:TheOtherCover}
$\TheCycle{n}$ is a cycle cover of size two in $\TheDigraph{n}$ when $n$ is even.
\end{remark}

\section{Concluding Remarks}

We proved that the directed sigma-tau $\sigma$-$\tau$ graph $\TheCayleyDigraph{n}$ has a Hamilton cycle $\TheCycle{n}$ when $n$ is odd, as well as a size two disjoint cycle cover $\TheCover{n}$ and Hamilton path $\ThePath{n}$ for all $n$.
This settles a longstanding problem in the well-studied area of Hamilton cycles in Cayley graphs.
Of particular interest is that $\TheCycle{n}$ and $\TheCover{n}$ are generated by simple rules, and that the underlying rotation systems are as close to trees as possible. 
In fact, the constructions were largely guided by these two goals. 
The simplicity of our local rules contrast the difficult recursive construction given by Compton and Williamson \cite{Compton} for the undirected sigma-tau graph $\TheCayleyGraph{n}$. 
%A companion article uses $\ThePath{n}$ to create a simple and efficient algorithm for generating permutations (see Williams \cite{Algorithm}).
%In fact, the algorithms would have fit nicely in the textbook that first examined the Hamiltonicity of $\TheCayleyDigraph{n}$ \cite{Wilf}.

The author thanks Alexander Holroyd (Microsoft Research) and Joe Sawada (University of Guelph) whose comments contributed to this article's presentation.

\begin{table}[h]
\begin{tabular}{|@{\,}c@{\,}|@{\,}c@{\,}|@{\,}c@{\,}||@{\,}c@{\,}|@{\,}c@{\,}|@{\,}c@{\,}|} \hline 
Symbol & Description & \S & Symbol & Description & \S \\ \hline \hline 
$\TheDigraph{n}$ & $\TheCayleyDigraph{n}$ & \ref{sec:intro} &
$\SigmaEdges$ & $\sigma$-edges in $\TheDigraph{n}$ & \ref{sec:intro} \\
$\TheCycle{n}$ & Hamilton cycle (odd $n$) & \ref{sec:intro} &
$\TauEdges$ & $\tau$-edges in $\TheDigraph{n}$ & \ref{sec:intro} \\
$\TheCover{n}$ & Cycle cover of size two & \ref{sec:intro} &
$\ThePath{n}$ & Hamilton path & \ref{sec:Hamilton_path}  \\
\hline
$\PERMS{n}$ & Permutation strings & \ref{sec:preliminaries_strings} &
$\SigmaCycles{n}$ & $\sigma$-cycle edges & \ref{sec:structure_basic} \\
$\MISSING[m]{n}$, $\MISSING{n}$ & Missing $m$, missing one & \ref{sec:preliminaries_strings} &
$\AltCycles{n}$ & Alternating-cycle edges & \ref{sec:structure_basic} \\
$\sim$ & Rotational equivalence & \ref{sec:preliminaries_equiv} & 
$s$ & Map $\quotient{\PERMS{n}}$ to $\SigmaCycles{n}$ & \ref{sec:structure_basic} \\
$\walkName$,$\edgesName$ & Walk, edges from $\sigma$ and $\tau$ & \ref{sec:preliminaries_walks} & 
$a$ & Map $\quotient{\MISSING{n}}$ to $\AltCycles{n}$ & \ref{sec:structure_basic} \\ 
\hline
$\FixSubset[n]{r}{m}$ & $\subseteq \quotient{\MISSING[m]{n}}$ with $r$ right of $n$ & \ref{sec:structure_covers} &
$\sink{S}{A}$ & Intersection of cycles & \ref{sec:rotation_Wilf} \\
$\YCycle{n}$ & $\subseteq \quotient{\MISSING{n}}$ related to $\TheCycle{n}$ & \ref{sec:structure_covers} &
$\ACycle{n}$ & Alt-cycles from $\YCycle{n}$ & \ref{sec:structure_covers} \\
$\YCover{n}$ & $\subseteq \quotient{\MISSING{n}}$ related to $\TheCover{n}$ & \ref{sec:structure_covers} &
$\ACover{n}$ & Alt-cycles from $\YCover{n}$ & \ref{sec:structure_covers} \\
\hline
$\Wilf{n}$ & Wilf's rotation system & \ref{sec:rotation_Wilf} &
$\Wheeel{n}$ & Rotation system & \ref{sec:rotation_Wheeel} \\
$\WilfCycle{n}$ & Induce $\Wilf{n}$ with $\YCycle{n}$ & \ref{sec:Hamilton} & 
$\WilfC{n}$ & Reduced $\WilfCycle{n}$ & \ref{sec:Hamilton_reduce} \\
$\WilfCover{n}$ & Induce $\Wilf{n}$ with $\YCover{n}$ & \ref{sec:Hamilton} & 
$\WilfP{n}$ & Reduced $\WilfCover{n}$ & \ref{sec:Hamilton_reduce} \\
\hline
\end{tabular}
\caption{Summary of notation and the section in which it is defined.}
\label{tab:summary}
\end{table}

\noindent Aaron Williams \\
\noindent Department of Mathematics and Statistics, McGill University \\
\noindent \texttt{haron@uvic.ca}

\end{document}